\documentclass[11pt,fleqn]{amsart} 

\usepackage[usenames,dvipsnames,condensed]{xcolor}

\usepackage{amsthm}
\usepackage{amsfonts}
\usepackage[english]{babel}
\usepackage[usenames]{xcolor}
\usepackage{graphicx}
\usepackage{soul}
\usepackage{stfloats}
\usepackage{morefloats}
\usepackage{cite}
\usepackage{lscape}
\usepackage{epstopdf}
\usepackage{braket}
\usepackage[lite]{amsrefs}
\usepackage{mathbbol}
\usepackage{tikz}
\usepackage{tikz-cd}

\usepackage{amsmath,amstext,amsopn,amsfonts,eucal,amssymb}
\usepackage{graphicx,wrapfig,url}

\newtheorem{theorem}{Theorem}[section]
\newtheorem{corollary}[theorem]{Corollary}
\newtheorem{lemma}[theorem]{Lemma}
\newtheorem{proposition}[theorem]{Proposition}
\newtheorem{definition}[theorem]{Definition}

\newtheorem{example}[theorem]{Example}

\newtheorem{remark}[theorem]{Remark}

\newcommand{\cal}{\mathcal}

\begin{document}

\title{Extensions of Augmented Racks and 
Surface Ribbon Cocycle Invariants}

\author{Masahico Saito} 
\address{Department of Mathematics, 
	University of South Florida, Tampa, FL 33620, U.S.A.} 
\email{saito@usf.edu} 

\author{Emanuele Zappala} 
\address{Yale School of Medicine, Yale University\\
	300 George Street, New Haven, CT 06511, U.S.A.} 
\email{emanuele.zappala@yale.edu \\ zae@usf.edu}

\begin{abstract}
A rack is a set with a  
binary operation 
that is right-invertible and self-distributive, properties diagrammatically corresponding to Reidemeister moves II and III, respectively. 
A rack is said to be an {\it augmented rack} if
 the operation is written by a group action.
Racks and their cohomology theories have been extensively used for knot 
and knotted surface
invariants. 
Similarly to group cohomology, rack 2-cocycles relate to extensions, 
and a natural question that arises is to characterize the extensions of augmented racks that are themselves augmented racks.
In this paper, we 
characterize such  extensions 
 in terms of what we call {\it fibrant and additive} cohomology of racks. 
 Simultaneous extensions of racks and groups are considered, where the respective $2$-cocycles are related through a certain formula.
 Furthermore, we construct coloring and cocycle  invariants for compact orientable surfaces with boundary 
in 
 ribbon forms  embedded in $3$-space.
\end{abstract}

\maketitle

\date{\empty}

\tableofcontents

\section{Introduction}

Self-distributive (binary) operations have been used since the 1980's to construct invariants of knots and links, following the articles \cite{Joyce,Mat}, where the notion of {\it fundamental quandle} was introduced,
defined topologically and diagrammatically. Homology and cohomology theories of quandles were  introduced, and used to construct 
 invariants of links in $3$-space, as well as knotted surfaces in $4$-space \cite{CJKLS}. These invariants are defined via certain state sum using quandle 2-cocycles, roughly described in the case of links in $3$-space as follows. The initial data of the construction is a quandle $X$, along with a $2$-cocycle of $X$ with coefficients in an abelian group $A$. First, one defines the set of $X$-colorings of a fixed diagram $D$ of a link $L$ as the set of homomorphisms from the fundamental quandle of $L$ (obtained through $D$) to $X$. 
 A coloring is also regarded as an assignment of elements of $X$ to arcs of $D$, and assigned elements are called {\it colors}.
  For each coloring, then, one takes the Boltzmann weights of each crossing of $D$, where the $2$-cocycle is evaluated,   $\phi(x_\tau, y_\tau)$,  where 
  $\phi$ is a 2-cocycle and $(x_\tau, y_\tau)$ is a certain pair of colors specified at the crossing $\tau$,
  see Figure~\ref{fig:positivenegative}.
   Then 
for each coloring,  
all these weights are multiplied together 
over all crossings, to obtain $\Pi_\tau \phi(x_\tau, y_\tau) \in A$.
Upon summing over all $X$-colorings, 
this quantity  $\sum_{\mathcal C} \Pi_\tau \phi(x_\tau, y_\tau) $ results to be invariant with respect to Reidemeister moves and, therefore, is independent of the choice of diagram of $D$.

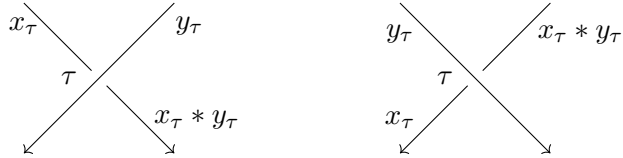
\begin{figure}[hbt]
	\begin{center}
		\begin{tikzpicture}
		\draw[->] (2,2) -- (0,0);
		\draw (0,2)--(0.9,1.1);
		\draw[->] (1.1,0.9) -- (2,0);
		
		\draw[->] (5,2)--(7,0); 
		\draw (7,2)--(6.1,1.1);
		\draw[->] (5.9,0.9)--(5,0);
		
		\node at (0,1.7) {$x_\tau$};
		\node  at (2.3,0.45) {$x_\tau*y_\tau$};
		\node at (2.2,1.7) {$y_\tau$};
		\node at (0.6,1) {$\tau$};
		
		\node  at (5,0.45) {$x_\tau$};
		\node  at  (5,1.6) {$y_\tau$};
		\node at (7.4,1.59) {$x_\tau*y_\tau$};
		\node at (5.6,1) {$\tau$};
		\end{tikzpicture}
	\end{center}
	\caption{Positive (left) and negative (right) crossings and their colorings for binary quandles.}
	\label{fig:positivenegative}
\end{figure}

Algebraically,   quandle 2-cocycles provide certain extensions of quandles in parallel to extensions by group 2-cocycles,
and there are bijective correspondences between equivalence classes of extensions 
and second quandle cohomology group \cite{CENS,CEStwisted}.
Relations and applications of algebraic theories of quandle extensions have been also obtained;
extensions provide 2-cocycles that are used for constructing cocycle invariants, and interpretations of cocycle invariants are provided in terms of obstructions of extending colorings to extensions
\cite{CENS,CEStwisted,ClaSa}.  Moreover, the relation between the extensions of certain quandles obtained from inner automorphisms of a group and the group itself has been studied in \cite{BaeCarterKim}, where homomorphisms between the cohomology groups are explicitly given.
 
These cocycle invariants have been generalized to invariants for  framed links via  {\it racks}
\cite{FR,SZframedlinks,EZ}, 
trivalent graphs for handlebody-links \cite{IIJO,CIST} and for 
 surface ribbons (orientable compact surfaces with boundary in the form of ribbons) embedded in 3-space 
  \cite{SZsfceribbon}.
In particular, in \cite{CIST,SZsfceribbon}, cocycle weights are also assigned to trivalent vertices, 
that are group 2-cocycles, and for the purpose of defining invariants, algebraic structures and cohomology theories that have both associative and self-distributive operations
in certain compatible manners have been developed
(see also \cite{CLY,Lebed}).

An {\it augmented rack} \cite{FR} is a rack $X$ with a map $\nu$ to a group $G$ acting on $X$, 
with certain conditions (see below).
We  could 
view augmented racks  as an algebraic structure with a self-distributive operation of the rack $X$ and the group structure of $G$ 
intertwined. 
From this point of view, we  provide two extension theories of augmented racks, and point out that the
corresponding 2-cocycles produce invariants for surface ribbons through trivalent graph
diagrams. 
As in \cite{CIST,SZsfceribbon}, rack cocycles are assigned to crossings and group cocycles are assigned to trivalent vertices. 
The purpose of this paper is threefold:
(1) define mixed cocycle conditions of cohomology groups in dimension 2
for group and rack 2-cocycles,
(2) provide algebraic 
 characterizations
 of these mixed cocycle conditions in terms of extensions, and
(3) provide a definition of surface ribbon cocycle invariants via trivalent graphs utilizing the mixed 2-cocycles at crossings and trivalent vertices. 
For these goals, we focus on 2-cocycles of racks and groups,
 instead of aiming to formulate a general homology theory. 

The paper is organized as follows.
In Section~\ref{sec:pre}, preliminary materials necessary to this paper are reviewed.
Extensions of augmented racks and properties of  their 2-cocycles are studied in Section~\ref{sec:ext},  and examples of extensions and cocycles are described.
In Section~\ref{sec:inv}  the construction of invariants  using these extension cocycles under certain additional conditions  is given.

\section{Preliminaries}\label{sec:pre}

In this section we review materials used in this paper.

\subsection{Augmented racks}

A {\it rack} is a set $X$ with a binary operation $(x, y ) \mapsto x*y =: S_y(x)$, where $S_y: X \rightarrow X$ is regarded as a map associated to $y$,  such that $S_y$ is bijective for all $y \in X$, and  $*$ is
right self-distributive, $(x*y)*z=(x*z)*(y*z)$ for all $x,y,z \in X$.
 It follows that $S_y$ is a rack isomorphism, and $X$ is called {\it connected} if  the subgroup of the rack automorphism group generated by $\{ S_y : y \in X\}$, called the inner automorphism group, acts transitively on $X$.
It follows that $x\ \bar{*}\ y:= S_y^{-1} (x)$ is also a rack operation on $X$.
A rack $X$ is called a {\it quandle} if it satisfies $x*x=x$ for all $x \in X$.

An {\it augmented rack}~\cite{FR} $(X, G)$ is a set $X$ with a right group action by  a group $G$ 
and a map $\nu: X \rightarrow G$ satisfying the identity
$\nu (x \cdot g) = g^{-1} \nu (x) g$ for all $x \in X$, $g \in G$, where  the symbol $\cdot$ indicates the group action.
An augmented rack has a rack operation defined by $x*y=x \cdot  \nu(y)$ for $x, y \in X$.
It is also said that $(X, *)$ is a $G$-rack.
It follows that  $x \ \bar{*} \ y = x \cdot \nu (y)^{-1}$.

\subsection{Good involutions}

Let $(X,*)$ be a rack.  The following definitions  are from \cite{Kamada}.
A {\it good involution} is an involution $\rho: X \rightarrow X$ that satisfies
$x * \rho(y)=x \ \bar{*} \ y$ and $\rho(x * y)=\rho(x)$ for all $x, y \in X$.
A rack with a good involution is called a {\it symmetric rack}.
If $X$ is a $G$-rack, then it is a symmetric rack with a good involution defined by 
$\nu( \rho (x) ) = \nu(x)^{-1}$.

\subsection{Group and rack  2-cocycles}

In this paper we focus on 2-cocycles and corresponding extensions of groups and racks, 
so that we review these materials 
in addition to group and rack cohomology theories. References include \cite{Brown,CJKLS}.

Let $G$ be a group and $A$ an abelian group.  
The $n^{\rm th}$ {\it cochain group}  (for the group $G$) is the set of functions $ G^n \rightarrow A$ under pointwise addition, and  is denoted by $C^n_G(G,A)$.
The coboundary operator $\delta^n_G: C^n_G(X,A) \rightarrow  C^{n+1}_{G} (X,A)$  (with trivial action on $A$)
is defined by 
\begin{eqnarray*}
\lefteqn{ ( \delta^n_G f) (x_1, x_2, \ldots, x_{n+1}) =}\\
& & f(x_2, \ldots, x_{n+1}) 
+  \sum_{i=1}^{n} (-1)^i  f (x_1,   \ldots, x_{i-1}, \widehat{x_{i}},  x_{i}  x_{i+1}, x_{i+2}, \ldots  , x_{n+1}) 
+ (-1)^{n+1} f(x_1, \ldots, x_n)
\end{eqnarray*}
for $f \in C^n_G(G,A)$. 
The cocycle group, coboundary group, cohomoligy groups are denoted 
as usually by $Z^n_G(G,A)$, $B^n_G(G,A)$ and $H^n_G(G,A)$, respectively.
For $n=1,2$, the differentials are formulated as 
\begin{eqnarray*}
(\delta^1_G \zeta )( x,y) &=&  \zeta(y) - \zeta(xy) +\zeta (x) ,  \\
(\delta^2_G \eta) (x,y,z) &=& \eta(y,z) + \eta(x, yz) -  \eta(xy,z) -  \eta(x,y) 
\end{eqnarray*}
for $\zeta \in C^1_G(X,A)$ and $\eta \in C^2_G(X,A)$. 

Let $\eta\in Z^2_G(X,A)$. Then $G \times A$ is endowed with a group structure by 
$$(x,a) (y,b) := (xy, a+b+ \eta(x,y))$$
 for all $(x,a), (y,b) \in G \times A$.
This group is called the {\it (central) group extension of $G$ by $A$ with respect to $\eta$}. 

The group 2-cocycle  condition has a well known diagrammatic interpretation as 
triangulation of a square as depicted in Figure~\ref{assoc}.
Three sides of a square are labeled by group elements, and each triangle receives a group 2-cocycle
evaluated by two sides of a triangle. As the figure shows, the two triangulations give rise to the 2-cocycle condition, and corresponds to the associativity, that ensures the extensions to be groups.

\begin{figure}[htb]
\begin{center}
\includegraphics[width=2.8in]{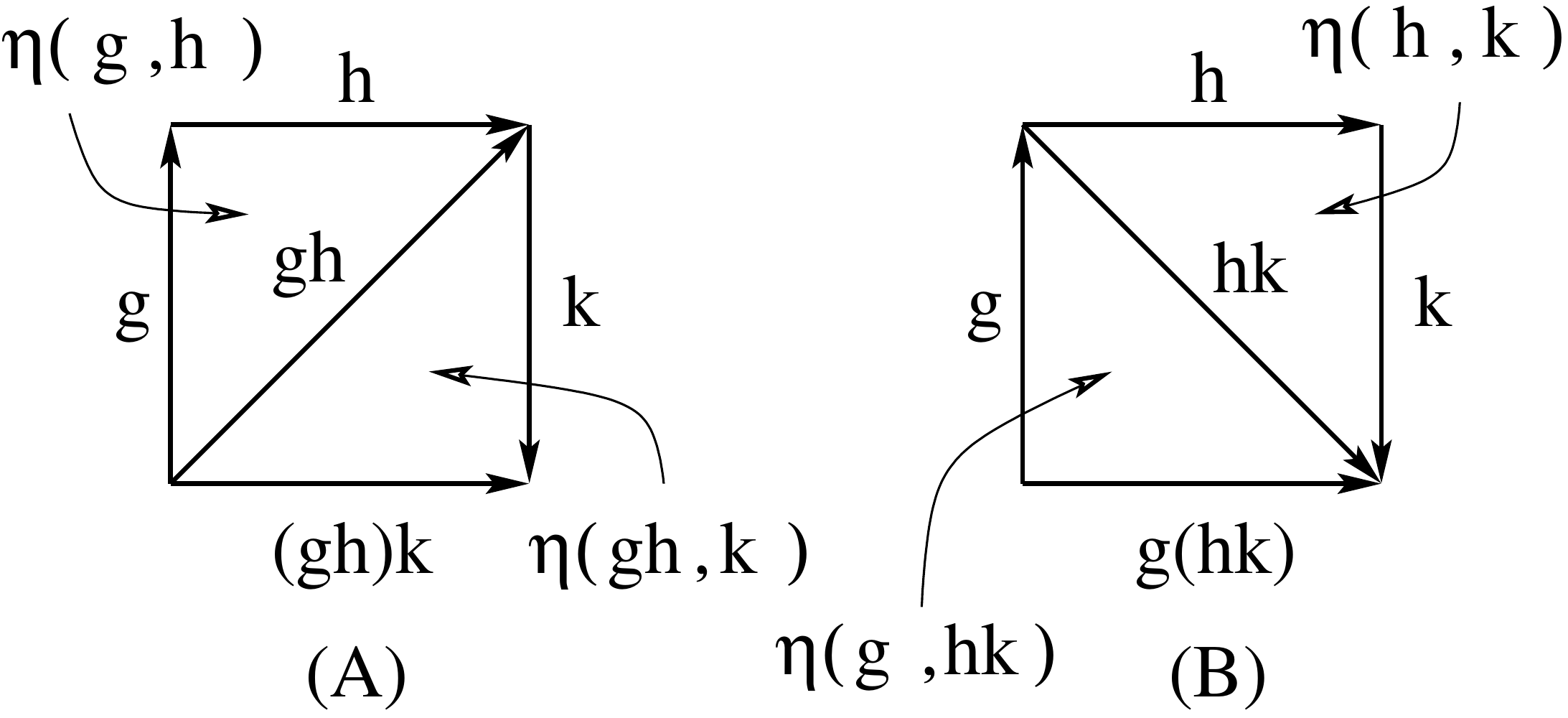}
\end{center}
\caption{Group 2-cocycle condition, triangulations of squares and associativity}
\label{assoc}
\end{figure}

It is computed from the 2-cocycle condition $(\delta^2_G \eta)=0$ that a 2-cocycle satisfies 
$\eta(g,e)=\eta(e,g)$ and $\eta(g,g^{-1})=\eta(g^{-1}, g) $ for $g\in G$ and $e\in G$ is the identity.
These also imply that the identity of the extension $G \times A$ is $(e, - \eta(e,e))$, 
and that $(g,s)^{-1}=(g^{-1}, -s - \eta(e,e))$ for $g \in G$, $s \in A$. 
A group 2-cocycle $\eta$ is called {\it normalized} if it satisfies $\eta(e,e)=0$. Then it follows that 
$\eta(g,e)=\eta(e,g)=\eta(g,g^{-1})=0$ for all $g \in G$. It is known that any second cohomology class has a normalized 2-cocycle representative.

Let $(X,*)$ be a rack and $A$ be an abelian group, and denote the rack operation by $*$.
 The $n^{\rm th}$ {\it cochain group}  (for the rack $X$) is the set of functions $ X^n \rightarrow A$ under pointwise addition, and  is denoted by $C^n_R(G,A)$.
The coboundary operator $\delta^n_R: C^n_{\rm R} (X,A) \rightarrow  C^{n+1}_R (X,A)$ 
is defined by 
\begin{eqnarray*}
\lefteqn{ ( \delta^n_R g) (x_1, x_2, \ldots, x_{n+1 }) =}\\
& &  \sum_{i=1}^{n} (-1)^i [\ (x_1, \ldots, \widehat{x_{i}},\ldots, x_{n+1}) 
- (x_1 *x_i,  \ldots, x_{i-1}*x_i,   \widehat{x_{i}}, x_{i+1}, \ldots  , x_{n+1}) \ ] 
\end{eqnarray*}
for $g \in C^n_R (X,A)$. The cocycle group, coboundary group, cohomoligy groups are denoted 
as usually by $Z^n_R(X,A)$, $B^n_R(X,A)$ and $H^n_R(X,A)$, respectively.
For $n=1,2$, the differentials are formulated as 
\begin{eqnarray*}
(\delta^1_R \xi )( x,y) &=& \xi (x) - \xi(x*y)  ,  \\
(\delta^2_R \phi) (x,y,z) &=& \phi(x,z)  - \phi (x*y,z)  - \phi (x, y) + \phi(x*z, y*z)
\end{eqnarray*}
for $\xi \in C^1_G(X,A)$ and $\phi \in C^2_G(X,A)$.

Let $\phi \in Z^2_R(X,A)$. Then $X \times A$ is endowed with a rack structure by 
$$(x,a)* (y,b) := (xy, a+ \phi(x,y))$$
 for all $(x,a), (y,b) \in G \times A$.
This rack is called the {\it rack extension of $G$ by $A$ with respect to $\phi$}. 

A rack 2-cocycle $\phi \in Z^2_{R}(X,A)$ for a symmetric rack $X$ with a good involution $\rho$ 
and a coefficient abelian group $A$ is called  {\it symmetric} \cite{Kamada} 
if it satisfies 
$$\phi(x, y ) + \phi( \rho(x), y)=0\quad {\rm  and} \quad 
\phi(x,y) + \phi (x*y, \rho(y))=0$$
 for all $x,y \in X$.

\begin{figure}[htb]
\begin{center}
\includegraphics[width=4in]{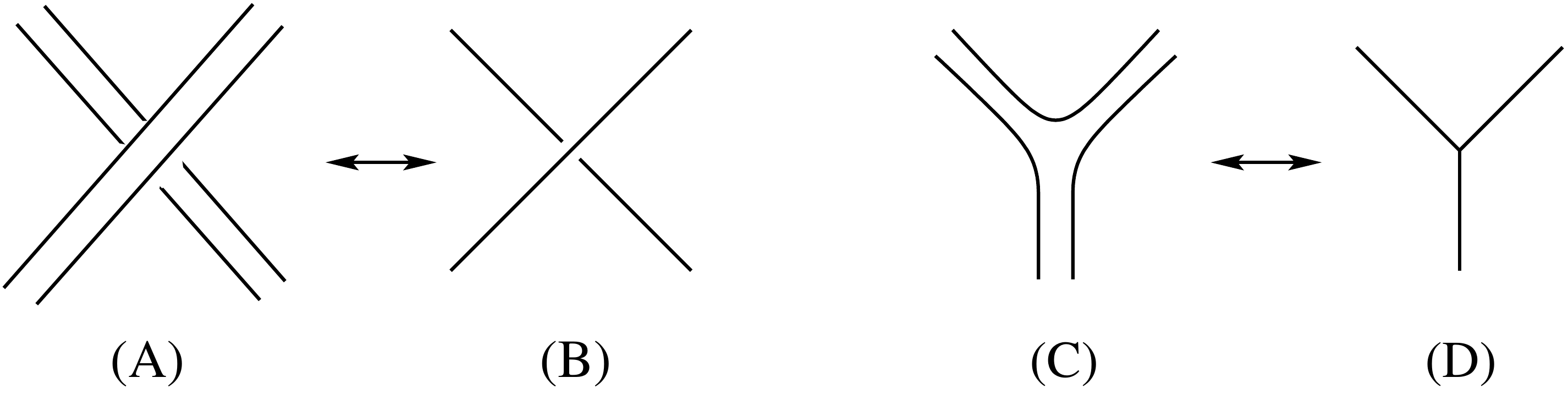}
\end{center}
\caption{Building blocks}
\label{blocks}
\end{figure}

\subsection{Diagrams of surface ribbons  and their moves}

In this section we review diagrams representing compact orientable surfaces with boundary embedded in 3-space (spatial surfaces with boundary). 
 Our discussion is based on \cite{Matsu}. By compact surfaces with boundary, we mean surfaces that are compact and such that each component has a non-empy boundary. 
 Such surfaces 
 are determined by their {\it spines} 
 and framing. 
 Recall that a spine for such a surface $S$ is a trivalent graph $G$ in $S$ such that a closed regular 
 neighborhood  of $G$
in $S$ 
 is a deformation retract of $S$. 
 We therefore represent compact surfaces with boundary, diagrammatically, as  fattened 
 trivalent graphs in ribbon forms.
      We call such representations {\it surface ribbons}.    Thus a surface ribbon is a compact orientable surface with boundary in the form of a thickened  flat trivalent graph. 
      When surface ribbons are embedded in 3-space, we consider its planar diagrams as for knot diagrams. 
Figure~\ref{blocks} (A) and (C) indicate local diagrams of such surface ribbons. 

\begin{figure}[htb]
\begin{center}
\includegraphics[width=1.2in]{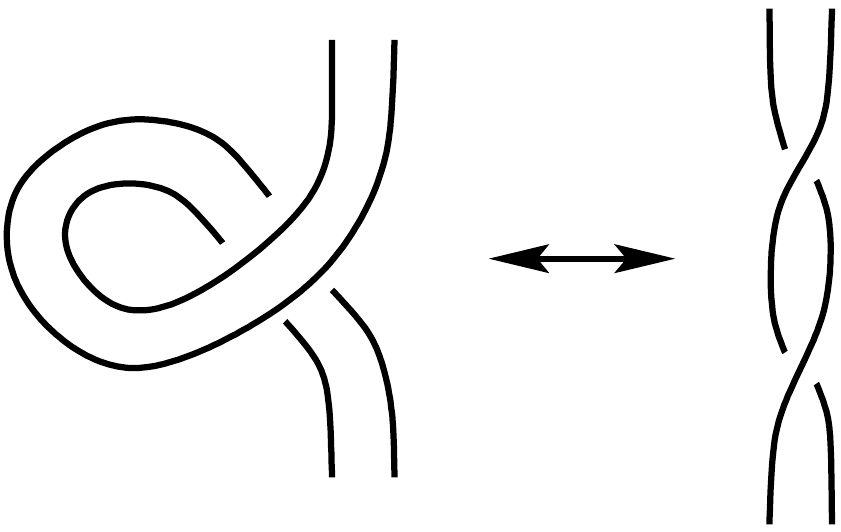}
\end{center}
\caption{A loop corresponds to a full twist}
\label{loop}
\end{figure}

     We further simplify surface ribbons to their spine, trivalent graphs, as depicted in Figure~\ref{blocks} (B) and (D). 
 The ribbons are assumed to be specified by {\it blackboard framing}, 
 where the arcs are fattened to ribbon forms parallel to the plane of projection, as known in framed knot diagrams.
  We focus on orientable surface ribbons.  In this case we can assume that no half-twist 
 occurs in diagrams, and    two half-twists 
  are represented by a small loop as in Figure~\ref{loop}.


\begin{figure}[htb]
\begin{center}
\includegraphics[width=3in]{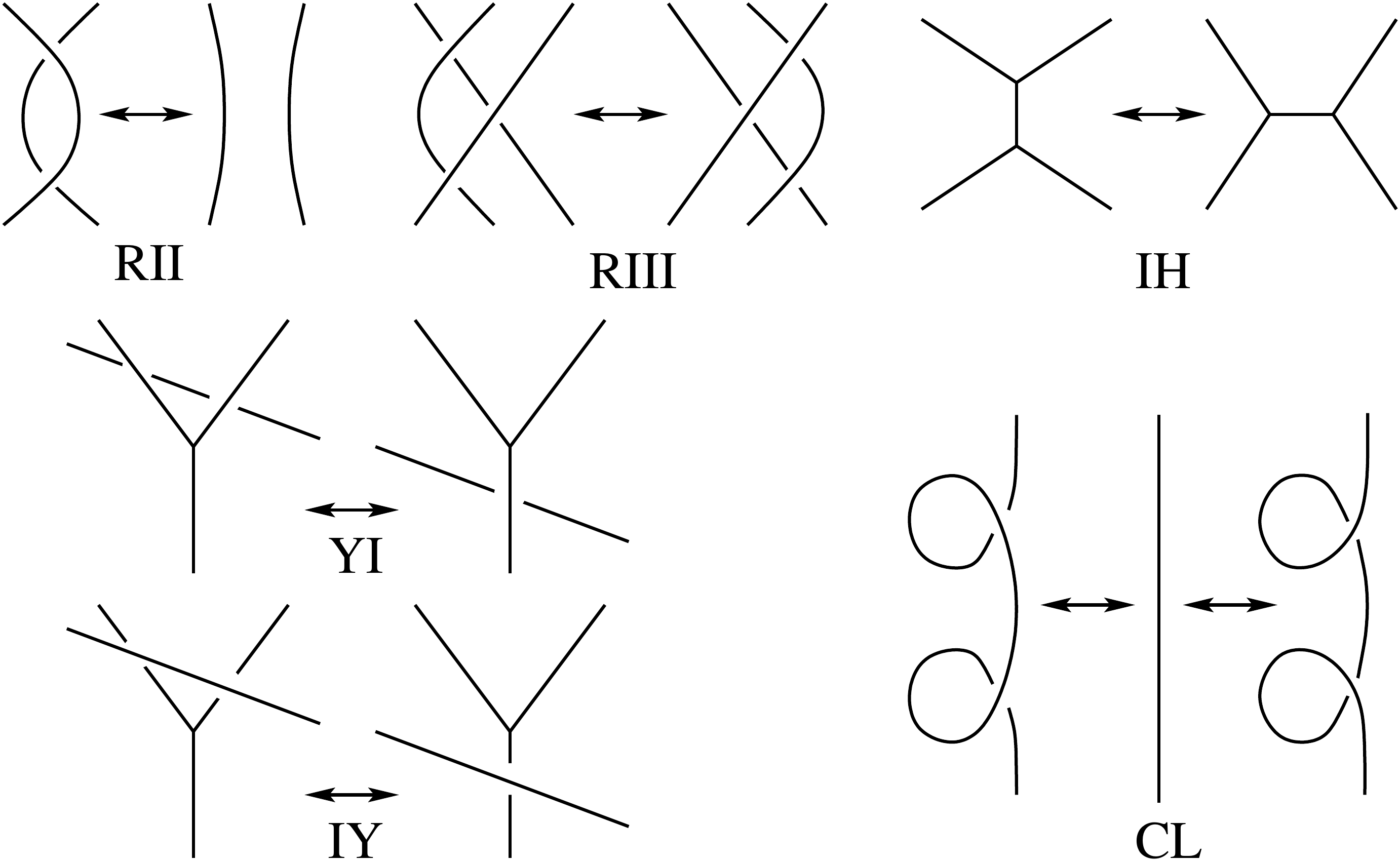}
\end{center}
\caption{Moves}
\label{moves}
\end{figure}

In \cite{Matsu}, it was shown that the isotopy class of a compact orientable surface with boundary
in a surface ribbon form 
 is determined diagrammatically by the moves given in Figure~\ref{moves}. Moves RII, RIII and CL are the framed Reidemeister moves for framed links. Moves IY, YI and IH appear also in the study of handlebody knots in $3$-space, see for instance \cite{Ishii08}. In particular, we mention that the IH move is important in the well posedness of the diagrammatic interpretation in terms of trivalent graphs (spines), since it allows to arbitrarily desingularize higher order vertices, so that we can consider only trivalent vertices as building blocks.
%
Matsuzaki has determined the moves for non-oriented surfaces as well \cite{Matsu}, although we do not consider this case here. 

\section{Extensions of augmented racks and second cohomology groups}\label{sec:ext}
\label{sec:rack_group}

In this section we consider 
two types of extensions of augmented racks and corresponding  constraints on cocycles.
Equivalences of extensions are defined, and bijections to certain second cohomology groups are established.


\subsection{Extensions of augmented racks with a fixed acting group} \label{subsec:Xext}

In this section we consider extansions of augmented racks with fixed acting groups. 

\begin{definition}\label{def:additive}
{\rm Let $(X,G)$ be an augmented rack, and $A$ an abelian group.
 A rack  2-cocycle $\phi \in Z^2_R(X, A)$ is said to be 
 {\it additive (with respect to the second factor)} if it satisfies 
 $$\phi (w, x) + \phi(w*x,y)=\phi(w,z)$$
  for all $w,x,y,z \in X$ such that 
 $\nu(x)\nu(y)=\nu(z)$. 
 }
 \end{definition}

%
%

\begin{lemma}\label{lem:delta_additive}
	Let $(X,G)$ be an augmented rack, and let $\xi :X \rightarrow A$ be a rack $1$-cochain. Then $\delta^1_R \xi$ is an additive $2$-cocycle. 
	Moreover, the sum of two additive cocycles is additive.
\end{lemma}
\begin{proof}
	For $x,y,z,w\in X$ such that $\nu(x)\nu(y) = \nu(z)$ we have the equalities 
	\begin{eqnarray*}
	(\delta^1_R \xi ) (w,x) + ( \delta^1_R \xi )( w*x,y) &=& \xi(w) - \xi(w*x) + \xi(w*x) - \xi((w*x)*y) \\
	&=& \xi(w) - \xi(w\cdot (\nu(x)\nu(y)))\\
	&=& \xi(w) - \xi(w*z)\\
	&=& ( \delta^1_R \xi ) (w,z). 
	\end{eqnarray*}
	The second part of the statement is immediate.
\end{proof}

As a consequence of Lemma~\ref{lem:delta_additive}, the following definition is well posed.

\begin{definition}\label{def:additive_cohomology}
	{\rm 
	Let $(X,G)$ be an augmented rack, and $A$ be an abelian group.
	 Then, the subgroup of $C^2_R(X,A)$ 
 (resp. $Z^2_R(X,A)$) consisting of additive 2-cochains (resp. 2-cocycles) is called the {\it additive (2-)cochain/cocycle  group},
 and  denoted by $C^2_{R+}(X,A)$ (resp. $Z^2_{R+}(X,A)$).
 The quotient $Z^2_{R+}(X,A)/  {\rm Im}( \delta_R^1 )$ is called 
 the {\it additive (second) cohomology group} of $X$, and denoted by $H^2_{R+}(X,A)$. 
}
\end{definition}

	\begin{definition}\label{def:fibrant}
	{\rm
		Let $(X,G)$ be an augmented rack, and let $\phi \in C^2_R(X,A)$ be a rack $2$-cochain with coefficients in $A$. Then, $\phi$ is said to be {\it $G$-fibrant}, or only {\it fibrant} for short, if it is constant on the fibers of $\nu$ with respect to the second entry. 
		In other words, for all $x , y, z \in X$ such that $\nu(y)=\nu(z)  \in G$, 
		it holds that $\phi(x,y) = \phi(x,z)$. In this case, $\phi$ induces a well defined map 
	$\phi : X\times {\rm Im}(\nu)  \rightarrow A$, 	which we denote by the same symbol,
	$\phi (x, g) = \phi(x, y)$ where $\nu(y)=g \in G$. 
	
	The fibrant $1$-cochains are the maps $\xi: X\rightarrow A$ that are constant on 
	preimages of $\nu$, i.e., for all $x, y \in X$ such that $\nu(x)=\nu(y)$, it holds that $\xi (x)=\xi (y)$. 
	}
	\end{definition}


	\begin{lemma}
		Let $(X,G)$ be an augmented rack, and let $\xi: X\rightarrow A$ be a 
		fibrant
		$1$-cochain with coefficients in an abelian group $A$. Then $\delta^1 \xi$ is fibrant. Moreover, the sum of two fibrant cocycles is fibrant.
	\end{lemma}
	\begin{proof}
		For $x,y,z\in X$, with $y, z\in \nu^{-1}(g)$ for some $g\in G$, we have 
		$$
		(\delta^1_R \xi) (x,y) = \xi(x) - \xi(x*y) = \xi(x) - \xi(x*z) = ( \delta^1_R \xi ) (x,z),
		$$
		if and only if $\xi(x*y) = \xi(x*z)$, which holds true, since $x*y = x\cdot g = x*z$. The second part of the statement is immediate.
	\end{proof}

	\begin{definition}
		{\rm 
			Let $(X,G)$ be an augmented rack and let $A$ be an abelian group. Then, we define the {\it  fibrant second rack cochain (resp. cocycle)  group}  of $X$ 
			with coefficients in $A$, denoted by $C^2_{RF} (X,A)$ (resp. $Z^2_{RF} (X,A)$),
			to be the subgroup of $C^2_R(X,A)$ (resp. $Z^2_R(X,A)$)
			that consists of 
			fibrant rack 2-cochains (resp.  cocycles).
			Similarly the first fibrant cochain group $C^1_{RF}(X,A)$ is defined.
	The corresponding cohomology group
	$Z^2_{RF} (X,A) / {\rm Im}(\delta_R^1  ( C^1_{RF} (X,A) )  ) $ is called the {\it  fibrant second rack cohomology  group}, and denoted by 
			$H^2_{RF} (X,A)$. 
		}
	\end{definition}

	\begin{definition}\label{def:fibadditive}
		{\rm 
		Let $(X,G)$ be an augmented rack, and let $A$ be an abelian group. Then, the {\it fibrant-additive} $2$-cochains (resp. cocycles) of $(X,G)$ with coefficients in $A$ are defined to be $2$-cochains (resp. cocycles) of $X$ that are both fibrant, and additive. They constitute a subgroup of $C^2_R(X,A)$ (resp. $Z^2_R(X,A)$) which is denoted by the symbol 
	$C^2_{RF+}(X,A)= C^2_{RF}(X,A) \cap C^2_{R+}(X,A) $ (resp.	$Z^2_{RF+}(X,A)= Z^2_{RF}(X,A) \cap Z^2_{R+}(X,A) $).
		
		The quotient of $Z^2_{RF+}(X,A)$ by the  fibrant rack coboundaries is called the {\it fibrant-additive cohomology} of $(X,G)$ with coefficients in $A$, and it is denoted by the symbol 
		$$H^2_{RF+}(X,A)=[ Z^2_{RF}(X,A) \cap Z^2_{R+}(X,A) ]/ {\rm Im} (\delta_R^1 ( C^1_{RF} (X,A) ) )  .$$
	}
	\end{definition}

We show how fibrant-additivity relates to abelian extensions.

\begin{proposition}\label{prop:Xext}
Let $X$ be a $G$-rack, and $A$ an abelian group.
Let $\tilde{X}=X \times A$ be an abelian extension by $\phi \in C^2_{RF+}(X,A)$, where $\phi$ is a fibrant and additive rack $2$-cochain. 
%
%
Define $\tilde{\nu}: \tilde{X} \rightarrow G$ by $\tilde{\nu} (x,a) = \nu(x)$,
and the action $\tilde{X} \times G \rightarrow \tilde{X}$ by 
$(x,a) \cdot g:= (x \cdot g, a + \phi(x, g) )$.
Then $(\tilde{X}, G)$ is a $G$-rack. 
\end{proposition}

\begin{proof}
	First we note that the additivity of $\phi$ under the assumption of being fibrant is reformulated,
by setting $\nu(x)=g$ and $\nu(y)=h$,  as
$\phi (w, g) + \phi (w \cdot g , h)=\phi(w, gh)$ from
$\phi (w, x) + \phi(w*x,y)=\phi(w,z)$, since $w*x=w\cdot \nu(x)=w \cdot g$ and $\nu(x)\nu(y)=gh=\nu(z)$. 

The action defined is indeed a right $G$-action:
for all $w, x,y,z \in X$ such that $\nu(s) \nu(y) = \nu (z)$, we have 
$ [  (w,d) \cdot g ] \cdot h =  (w \cdot g, d + \phi(w, g) ) \cdot h
= ( [ w \cdot g] \cdot h , d + \phi(w, g) + \phi  (w \cdot g, h)  ) $,
which is equal to 
$  (w,d) \cdot ( g  h )  = (w \cdot (gh) , d + \phi (x, gh) )  $ by the above reformulated additivity.

Then one checks the $G$-rack condition 
$\tilde{\nu}( (x,a) \cdot g) = \tilde{\nu}( x \cdot g , a + \phi( x, g)) =
\nu(x \cdot g ) =g^{-1} \nu(x) g = 
g^{-1} \tilde{\nu}( x,a) g$ as desired. 
\end{proof}

\begin{definition}
{\rm
Let $X$ be a $G$-rack, and $A$ an abelian group, and $\phi \in Z^2_{RF+}(X,A)$.
The $G$-rack $\tilde X=X \times A$ defined in Proposition~\ref{prop:Xext} by $\phi$ is called
a {\it $G$-rack extension} of $(X,G)$ by $\phi$.
}
\end{definition}

In Proposition~\ref{prop:Xext}, extensions are defined by 2-cochains. 
We show, in fact, that  $2$-cochains automatically are 2-cocycles.
This gives a method of constructing 2-cocycles.

\begin{proposition}\label{prop:fib+cocy}
Any fibrant-additive 2-cochain is a 2-cocycle: 
$C^2_{RF+}(X,A) 
=Z^2_{\rm RF+} (X,A)$.
\end{proposition}

\begin{proof}
Let $\phi \in C^2_{RF+}(X,A) $ as in Proposition~\ref{prop:Xext}.
Note that the rack operation is written as 
$$(x,a)*(y,b)=(x,a) \cdot \tilde{\nu} (y, b)=(x,a) \cdot ({\nu} (y) , b)=(x \cdot ({\nu} (y) , a +\phi(x, \nu(y)) )
=(x*y, a +\phi(x, \nu(y)) ) , $$
which is the original rack extension by a 2-cocycle.
Hence the original computation of the extension applies to obtain the rack 2-cocycle condition
from the self-distributivity:
\begin{eqnarray*}
 ( (x,a)*(y,b) )*(z,c) &=& ( (x*y)*z, a+\phi(x,y) + \phi(x*y, z) ) \\
 ((x,a)*(z,c) )* ((y,b)* (z,c) )&=& ( (x*z)*(y*z), a + \phi (x,z) + \phi( x*z, y*z) ) ,
 \end{eqnarray*}
 as expected.
\end{proof}


We provide constructions of examples using group extensions.
Let $G$ be a group and $Q$ a union of its conjugacy classes.
Then $Q$ is an augmented quandle by $\nu: Q \rightarrow G$ the inclusion, and is a conjugation subquandle of $G$.

Let $1 \rightarrow A \rightarrow  \tilde G \stackrel{p}{\rightarrow} G \rightarrow 1$ be a 
central extension by $\eta \in Z^2_G(G,A)$. 
Let $\tilde Q \subset \tilde G$ be a union of conjugacy classes, and $Q:=p(\tilde Q)$.
Then $\tilde Q$ is a conjugation subquandles of  $\tilde G$.
We denote the conjugation quandle operation by $*$ for $G$, $Q$ and $\tilde *$ for $\tilde G$, $\tilde Q$, respectively.

\begin{lemma}
Let  $G$, $Q$, $\tilde G$, and $\tilde Q$ be as above.
Let $s: G \rightarrow \tilde G$ be a set-theoretic section.
Define an action of $G$ on $\tilde Q$ by $\tilde x \cdot g= s(g)^{-1} x s(g)$ for $\tilde x \in \tilde X, g \in G$.
Then $\tilde Q$ is an augmented rack with $\tilde\nu  := \nu \circ p$.
\end{lemma}

\begin{proof}
For $g, h \in G$, there exists $a \in A $ such that $s(g) s(h) = a s(gh)$, where 
$A$ is regarded as a subgroup in the center $Z(\tilde G)$.
For $\tilde x  \in \tilde X$ and $g,h \in G$, one computes 
$$ (\tilde x \cdot g) \cdot h = s(h)^{-1} s(g)^{-1} \tilde x s(g) s(h) = (a s(gh))^{-1} \tilde x (a s(gh) )
=s(gh))^{-1} \tilde x s(gh) =\tilde x \cdot (gh),$$
so that this defined an action.
One computes 
$\tilde \nu( \tilde x \cdot g)= \tilde \nu( s(g)^{-1} \tilde x s(g) ) = \nu ( g^{-1} p (\tilde x) g) $
for $\tilde x \in \tilde Q$ and $g \in G$, as desired.
\end{proof}

Note that in the above situation, $\tilde \nu (\tilde Q)=Q$. 
In \cite{ClaSa}, it is proved that 
if  $|\tilde Q |/|Q|=2$, then $\tilde Q$ is an abelian extension of $Q$.
Hence the preceding lemma gives rise to examples of 
Proposition~\ref{prop:Xext}. However, we do not know when the corresponding rack cocycles are additive.

Next we establish a bijection between equivalence classes of $G$-rack extensions and 
the second fibrant-additive cohomology group  $H^2_{RF+} (X,A) $.

\begin{definition}\label{def:Eq_Xext}
{\rm
Let $X$ be a $G$-rack and $A$ an abelian group.
Let $\tilde{X}_i=X \times A$, $i=1,2$, be $G$-rack extensions of a $G$-rack $X$ by $\phi_i \in Z^2_{RF+}(X,A)$
as defined in Proposition~\ref{prop:Xext}.
Let $p_i : \tilde{X}_i \rightarrow X$ be the projection to the first factor. 
We say that 
$( \tilde{X}_i, p_i) $ 
are equivalent if 
there is a 
bijection 
$F : \tilde{X}_1  \rightarrow \tilde{X}_2$ such that the following diagrams commute. \\
\begin{center}
		\begin{tikzcd}
		\tilde X_1 \arrow[rr,"F"]\arrow[rd,"  p_1" swap] && \tilde X_2\arrow[ld," p_2"]\\
		   &X &\\[-15pt]    
		   &(1)&
		\end{tikzcd}
%
%
\quad			\begin{tikzcd}
				\tilde X_1\times G \arrow[rr]\arrow[d," F \times {\rm id}_G" swap] &&\tilde X_1\arrow[d,"F"] \\
				\tilde X_2\times G \arrow[rr] && \tilde X_2 
				\\[-15pt]  &(2)& 
				\end{tikzcd}
		\end{center}
If such a 
 map 
$F$ exists, we say that $F$ is an {\it isomorphism of extensions}, or isomorphism for short.
}
\end{definition}

\begin{theorem}\label{thm:Xext}
Let $(X,G)$ be an augmented rack, and let $A$ be an abelian group. 
Then the equivalence classes of G-rack extensions of $(X,G)$ by $A$ are in bijective correspondence with the 
fibrant-additive rack cohomology group $H^2_{RF+}(X, A)$. 
\end{theorem}

\begin{proof}
Let $F : \tilde{X}_1  \rightarrow \tilde{X}_2$ be 
 a bijection 
 in the definition of equivalence between
extensions $p_i: \tilde{X}_i \rightarrow X$, $i=1,2$, by $\phi_i \in Z^2_{RF+}(X,A)$
as in 
Definition~\ref{def:Eq_Xext}. 
Since $p_1 = p_2 \circ F$ from the commutative diagram (1) in Definition~\ref{def:Eq_Xext}, for any $(x, a) \in \tilde{X}_1$ there exists $\xi(x) \in A$ such that 
$F(x,a)=(x, a+\xi(x)) \in \tilde{X}_2$. This defines a function $\xi: X \rightarrow A$, $\xi \in C^1_{R}(X,A)$.  
In addition, we observe that 
since $$(  \tilde \nu_2 \circ F) (x,a) = \tilde \nu_2(  x, a+ \xi(x) ) = \nu (x) = \tilde \nu_1 (x,a), $$
the diagram
	\begin{center}
		\begin{tikzcd}
		\tilde X_1 \arrow[rr,"F"]\arrow[rd,"  \tilde \nu_1" swap] && \tilde X_2\arrow[ld," \tilde \nu_2"]\\
		   &G &   
			   \\
			\end{tikzcd}
		\end{center}
commutes 
as well, establishing that $F$ lies over $G$. 
For $(x,a), (y,b) \in \tilde{X}_1$, we have 
\begin{eqnarray*}
F( (x,a)*(y,b) ) &=& ( x*y, a+\phi_1(x,y) + \xi(x*y) ) , \\
F( x,a ) * F ( y,b ) &=& (x, a+\xi(x)) * (y, b+\xi(y)) \ =\  (x*y, a+\xi(x) + \phi_2 (x,y) ),
\end{eqnarray*}
from the commutative diagram (2).
Hence $\phi_1=\phi_2 + \delta^1_R \xi$, and we obtain $[\phi_1]=[\phi_2] \in H^2_{RF+}(X,A)$
as desired. 
Observe, in particular, that the previous equalities show that $F$ is also a rack homomorphism with respect to the rack extension structures. 
Conversely, if $\phi_1=\phi_2 + \delta^1_R \xi$, then $F((x,a)) := (x, a + \xi(x))$ defines a desired 
isomorphism.
\end{proof}

 \subsection{Simultaneous extensions of augmented racks}\label{subsec:Gext}

In this section we generalize the definition of G-rack extensions to extensions of $X$ and $G$ simultaneously, for augmented $G$-rack $X$ by an abelian group $A$, define appropriate cohomology, and establish a bijection
between equivalence classes of such extensions and the cohomology defined.

\begin{definition}\label{def:etaderived}
{\rm
Let $X$ be a $G$-rack,  $A$ an abelian group,
 $\phi \in Z^2_{R}(X,A)$, 
 and $\eta \in Z^2_G(G,A)$. 
We say that  $\phi $ is {\it derived} from $\eta$, or {\it $\eta$-derived}, 
if they satisfy 
$$\phi(x,y)=- \eta(e,e) 
 - \eta(\nu(y), \nu(y)^{-1}) + \eta(\nu(y)^{-1}, \nu (x)) + \eta(\nu(y)^{-1} \nu (x) , \nu(y)) $$
for all $x, y \in X$, 
where $e\in G$ is the identity.
Note that the first two negative terms vanish for normalized group 2-cocycles $\eta$. 
}
\end{definition}

\begin{remark}\label{rem:totallyfib}
{\rm
Note that if $\phi $ is $\eta$-derived in the preceding definition, 
then 
$\phi$ is {\it totally fibrant}, in the sense that both factors depend only on the image 
of $\nu$.
In particular, $\phi$ is fibrant, i.e. $\phi \in  Z^2_{RF}(X,A)$. 

If $\phi$ is totally fibrant, then it can be written as a pull-back $\phi=\nu^\sharp \bar{\phi}$, 
which means 
$\phi(x,y) =\bar{\phi} (\nu(x), \nu(y) )$ for some $\bar{\phi}: G \times G \rightarrow A$.


}
\end{remark}

\begin{example}\label{ex:central_extension}
{\rm

Let $(X,G)$ be an augmented rack and $A$ an abelian group.
Let $\tilde G = G \times A$ be a central  extension of $G$ by a 2-cocycle $\eta \in Z^2_G(X,A)$.
Let a function $\phi': G \times G \rightarrow A$ be defined by
$$\phi'(g,h)=- \eta(e,e) 
 - \eta(h, h^{-1}) + \eta(h^{-1},g) + \eta(h^{-1}g,h), $$
via the right-hand side of the equality in Definition~\ref{def:etaderived}. 
Let $\phi=\nu^\sharp (\phi')$ be the pull-back as described in Remark~\ref{rem:totallyfib}.
Thus we obtain $\eta$-derived cocycles by pull-backs. 

}
\end{example}

%

 \begin{proposition}\label{prop:Gext}
 Let $X$ be a $G$-rack, and $A$ an abelian group.
Let $\tilde{X}=X \times A$ be a rack  extension by $\phi \in Z^2_R(X,A)$,
and $\tilde{G}=X \times A$ be a central extension of $G$ by $\eta \in Z^2_G(G,A)$. 
Assume, further, that $\phi $  is additive ($\phi \in Z^2_{R+}(X,A)$), 
and  $\eta$-derived (so that $\phi$ is in fact fibrant-additive, $\phi \in Z^2_{RF+}(X,A)$).

Define $\tilde{\nu}: \tilde{X} \rightarrow \tilde{G}$ by $\tilde{\nu} (x,a) = (\nu(x), a)$,
and the action $\tilde{X} \times \tilde{G} \rightarrow \tilde{X}$ by 
$(x,a) \cdot (g, s) := (x \cdot g, a + \phi(x, g) )$.
Then $\tilde{X}$ 
 is a $\tilde{G}$-rack. 
\end{proposition}
 
 \begin{proof}
The well-definedness of the action is similar to the proof of Proposition~\ref{prop:Xext}.
Then one checks the $\tilde{G}$-rack condition:
$$\tilde{\nu}( (x,a) \cdot (g,s) ) = \tilde{\nu}( x \cdot g , a + \phi( x, g)) =
( \nu(x \cdot g ) , a + \phi( x, g)) = ( g^{-1} \nu(x) g ,  a + \phi( x, g)), $$
and 
\begin{eqnarray*}
\lefteqn{ (g, s)^{-1} \tilde{\nu} (x,a)  (g,s) }\\
&=& 
(g^{-1}, -s - \eta(e,e) - \eta(g,g^{-1} ) )  ( \nu (x),  a ) (g, s) \\
&=&
(g^{-1} \nu (x) , a  - s  - \eta(e,e) - \eta(g, g^{-1} )+ \eta(g^{-1}, \nu (x) ) ) (g,s) \\
&=&
(g^{-1} \nu (x) g , a  - s - \eta(e,e) - \eta(g, g^{-1}) + \eta(g^{-1}, \nu (x)) + s + \eta(g^{-1} \nu (x) , g) )
\end{eqnarray*}
so that the condition holds from the assumption that $\phi$ is $\eta$-derived. 
\end{proof}

%
%
%

\begin{corollary}\label{cor:compatibility}
Let $(X,G)$ be an augmented rack, and $A$ an abelian group.
Let $\phi \in Z^2_{RF+}(X,A)$ be fibrant-additive and $\eta$-derived.
		Then the equality
		$$
		\phi(x,w) + \phi(y,w) + 
		\eta(\nu(w)^{-1} \nu(x) \nu(w), \nu(w)^{-1} \nu(y) \nu(w) ) 
		= \phi(z,w)+ \eta(\nu(x), \nu(y) )
		$$
	holds 	for all $x,y,z,w \in X$ such that $\nu(x)\nu(y)=\nu(z)$.
\end{corollary}

\begin{proof}
We use the notation in Proposition~\ref{prop:Gext}. 
Note that 
$$\tilde{\nu}  (x, a) \tilde{\nu}  (y,b) = ( \nu (x) \nu (y) , a + b + \eta ( \nu (x) , \nu (y) ) ) .$$
If $\nu (x) \nu (y) = \nu (z)$, then we have
$$
( \nu (x), a) (\nu (y) , b)
=
( \nu (x) \nu (y) , a + b + \eta ( \nu (x) , \nu (y) ) )
=
\tilde{\nu } (z,  a + b + \eta ( \nu (x) , \nu (y) ))
 . $$
 Hence we have 
\begin{eqnarray*}
\lefteqn{\tilde{\nu} ( (x, a)*(w,d) ) \tilde{\nu} ( (y,b)*(w,d) )}\\
&=& \tilde{\nu} ( (x, a) \cdot \tilde{\nu}  (w,d) ) \tilde{\nu} ( (y,b)\cdot \tilde{\nu}  (w,d)  ) \\
&=& [  \tilde{\nu}  (w,d)^{-1} \tilde{\nu}  (x, a) \tilde{\nu}  (w,d) ]\ 
[ \tilde{\nu}  (w,d)^{-1} \tilde{\nu}  (y,b) \tilde{\nu}  (w,d) ] \\
&=& \tilde{\nu}  (w,d)^{-1}\ [  \tilde{\nu}  (x, a) \tilde{\nu}  (y,b) ] \ \tilde{\nu}  (w,d) \\
 &=& \tilde{\nu} ( (z, a + b + \eta ( \nu (x) , \nu (y) ) \cdot \tilde{\nu}  (w,d) )) \\
&=& \tilde{\nu} ( (z, a + b + \eta ( \nu (x) , \nu (y) ) * (w,d) )) \\
&=& (\nu (z*w) , a + b + \eta ( \nu (x) , \nu (y) ) + \phi(z, w)) .
\end{eqnarray*}
We also  obtain
\begin{eqnarray*}
\lefteqn{ \tilde{\nu} ( (x, a)*(w,d) ) \tilde{\nu} ( (y,b)*(w,d) ) }\\
&=& ( \nu(x*w) , a+\phi(x,w) ) ( \nu(y*w) , b+ \phi(y,w) ) \\
&=& ( \nu(x*w)  \nu(y*w) , a+\phi(x,w) + b+ \phi(y,w) + \eta( \nu(x*w) , \nu (y*w) )) ,
\end{eqnarray*}
which implies the equality as desired.
 \end{proof}

 \begin{definition}\label{def:aug_Ext}
	{\rm 
	Let $(X,G)$ be an augmented rack,   $A$ be an abelian group,
	and $\phi \in Z^2_{RF+}(X,A)$, $\eta \in Z^2_{G}(X, A)$, such that $\phi$ is $\eta$-derived.
	 Then, the extension defined in Proposition~\ref{prop:Gext}, $(X\times A, G\times A)$, is 
	 called an {\it augmented (simulteneous) extension} of $(X,G)$ by $(\phi, \eta)$.
	
}
\end{definition}

Next, we define an equivalence relation among augmented extensions.

\begin{definition}\label{def:Ext_morphism}
	{\rm 
Let $(\tilde X_i, \tilde G_i)$, for $i = 1,2$, denote two augmented extensions of the augmented rack $(X,G)$, where $\tilde X_i = X\times A$ and $\tilde G_i = G\times A$ as sets for both $i = 1,2$,
by $(\phi_i, \eta_i)$, where $\phi_i \in Z^2_{RF+}(X_i, A)$ and $\eta_i \in Z^2_{G}(G, A)$,
as defined in Proposition~\ref{prop:Gext}.
Denote the augmentation maps by $\tilde{\nu}_i : \tilde{X}_i \rightarrow \tilde G$.
%
%

Then, a morphism of augmented extensions is a pair of 
maps
  $F_X : \tilde{X}_1 \rightarrow \tilde X_2$ and $F_G: \tilde G_1 \rightarrow \tilde G_2$,  where $F_G$ is a group homomorphism, such that 
$p_1=p_2 \circ F_X$ and $p_1 = p_2 \circ F_G$ with respective projections $p_i$, $i=1,2$ in the same letter, 
satisfying the following diagrams commute.\\
\begin{center}
		\begin{tikzcd}
		\tilde X_1 \arrow[rr,"F"]\arrow[rd," p_1" swap] && \tilde X_2\arrow[ld," p_2"]\\
		   &X &   
		   \\[-15pt] &(1)& 
		\end{tikzcd}
\hspace{16mm}
	\begin{tikzcd}
		\tilde G_1 \arrow[rr,"F"]\arrow[rd," p_1" swap] && \tilde G_2\arrow[ld," p_2"]\\
		   &G &   
		      \\[-15pt]  &(2)& 
		\end{tikzcd} 
		\\
%
%
%
%
\quad			\begin{tikzcd}
				\tilde X_1\times \tilde G_1 \arrow[rr]\arrow[d,"F_X \times F_G" swap] &&\tilde X_1\arrow[d,"F_X"]\\
				\tilde X_2\times \tilde G_2 \arrow[rr] && \tilde X_2  
				   \\[-15pt]  &(3)& 
				\end{tikzcd}
		\end{center}
	An isomorphism, is a 
 morphism where both $F_X$ and $F_G$ are bijections. Observe that in this case the inverses automatically satisfy the commutativity of the previous diagrams, guaranteeing that $(F^{-1}_X,F^{-1}_G)$ is a morphism as well.
	}
\end{definition}


\begin{definition}\label{def:mixed-coh}
	{\rm 
	Let $(X,G)$ be an augmented rack, and let $A$ be an abelian group. We define the {\it augmented rack second cocycle group} with coefficients in $A$, denoted by $Z^2_{AR}((X,G),A)$, to be  the subgroup of the direct sum $ Z_{RF+}^2(X,A) \oplus Z^2_{GN}(G,A)$ generated by 
$$\{ \phi(x,y) + \eta(\nu(y), \nu^{-1}) - \eta(\nu(y)^{-1}, \nu (x)) - \eta(\nu(y)^{-1} \nu (x) , \nu(y)) 	
\mid x, y  \in X \} . $$
Then we define the {\it augmented rack second cohomology group} with coefficients in $A$, denoted by $H^2_{AR}((X,G),A)$, to be the quotient $Z^2_{AR}((X,G),A)/ {\rm Im} (\delta^1_{R} + \delta^1_{G} )$.
}
\end{definition}

\begin{theorem}
	Let $(X,G)$ be an augmented rack, and let $A$ be an abelian group. Then the isomorphism classes of augmented extensions of $(X,G)$ by $A$ with respect to cocycles $(\phi, \eta)$ are in bijection with 
	the augmented rack cohomology group $H^2_{AR}((X,G),A)$. 
\end{theorem}
\begin{proof}
Let $(\tilde X_i, \tilde G_i)$, for $i = 1,2$, denote two equivalent 
augmented extensions of the augmented rack $(X,G)$, where $\tilde X_i = X\times A$ and $\tilde G_i = G\times A$ as sets for both $i = 1,2$,
by $(\phi_i, \eta_i)$, where $\phi_i \in Z^2_{RF+}(X_i, A)$ and $\eta_i \in Z^2_{G}(G, A)$,
and denote the augmentation maps by $\tilde{\nu}_i : \tilde{X}_i \rightarrow \tilde G$,
as in Definition~\ref{def:Ext_morphism}, through isomorphisms $F_X: \tilde X_1 \rightarrow \tilde X_2$
and $F_G: \tilde G_1 \rightarrow \tilde G_2$.

By arguments similar to the proof of Theorem~\ref{thm:Xext},
from isomorphisms $F_X$ and $F_G$ in the commutative diagrams (1) and (2) in Definition~\ref{def:Ext_morphism},
 we have that
for all $(x,a) \in \tilde X_1$ and $(g,b) \in \tilde G_1$  there exist $\xi(x), \zeta (g) \in A$ such that 
$F_X ( x, a) = (x, a + \xi(x))$ and $F_G (g,b) = (g, b+ \zeta (g)) $. 
Commutativity of diagrams $(1)$ and $(3)$ implies that $F_X$ is a rack homomorphism, which combined with the fact that $F_G$ is a group homomorphism by hypothesis,  
implies that $[\phi_1]=[\phi_2]$ and 
$[\eta_1]=[\eta_2]$ through $\xi$ and $\zeta$. 
Since $(\tilde X_i, \tilde G_i)$ are augmented rack extensions through $(\phi_i, \eta_i)$, 
it also follows that $\phi_i$ is $\eta_i$-derived, for $i=1,2$. 
Hence $[ (\phi_1, \eta_1) ] = [ (\phi_2, \eta_2)] \in H^2_{AR}( (X,G),A)$.
Below we check that the rest of the commutative diagrams for isomorphisms do not impose 
additional constraint.

One computes, for $(x,a) \in \tilde X_1$, 
\begin{eqnarray*}
(\nu_2 \circ F_X)  (x,a) &=& \nu_2 ( x, a + \zeta(x) ) \ = \ (\nu (x) , a + \zeta(x) ) , \\
(F_G \circ \nu_1) (x,a) &=& F_G ( \nu_1( x,a ))  \ = \  F_G ( \nu (x) ,a ) \ = \ (\nu (x) , a + \zeta(x) ),
\end{eqnarray*}
hence (3) in Definition~\ref{def:Ext_morphism} commutes under the assumption. 

For (4),  one computes, for $(x,a) \in \tilde X_1$ and $(g,b) \in \tilde G_1$, 
\begin{eqnarray*}
F_X ( (x,a) \cdot (g,b) ) &=& F_X ( x \cdot g, a + \phi_1 (x, g) ) \ = \  (x \cdot g, a +  \phi_1 (x, g) + \xi(x \cdot g) ) , \\
 F_X (x,a) \cdot F_G (g,b) &=& (x, a + \xi (x) ) \cdot (g, b + \zeta (g) ) \ = \
 (x \cdot g, a + \xi(x) + \phi_2 (x, g) ) ,
 \end{eqnarray*}
so that we obtain 
$\phi_1 (x, g) =   \phi_2 (x, g) +\xi(x)  - \xi(x \cdot g)  = \phi_2 (x,g) + (\delta^1_R \xi ) (x, g)$,
and we obtain that $[\phi_1]=[\phi_2] \in Z^2_{RF+} (X,A)$. 
Thus an isomorphism implies $[ (\phi_1, \eta_1) ] = [ (\phi_2, \eta_2)] \in H^2_{AR}( (X,G),A)$.

Conversely, if $[(\phi_1, \eta_1) ] = [ (\phi_2, \eta_2)] \in H^2_{AR}( (X,G),A)$,
then there are 1-cochains $\xi \in C^1_{R}(X,A)$ and $ \zeta \in C^1_G(G,A)$
such that $\phi_1=\phi_2+\delta_R \xi$ and $\eta_1=\eta_2 + \delta^1_G \zeta$.
Define $F_X: \tilde X_1 \rightarrow \tilde X_2$ and $F_G: \tilde G_1 \rightarrow \tilde G_2$ 
by $F_X (x,a)=(x, a+\xi(x))$ and $F_G (g,b)=(g, b+\zeta(g))$, 
then it is readily checked that the diagrams (1) through (4) commute.
\end{proof}

%

\section{Colorings and cocycle invariants of 
surface ribbon 
diagrams by augmented racks}\label{sec:inv}

\subsection{Colorings of trivalent graphs representing surface ribbons}

Let $(X, G)$ be an augmented symmetric rack with a good involution $\rho$.
Let $D$ be a graph diagram of a surface ribbon $S$, and ${\mathcal A}(D)$ the set of 
arcs of $D$. 
Let $\vec{D}$ denote $D$ with an 
orientation specified, and 
${\mathcal A}(\vec{D})$ the set of directed arcs of $\vec{D}$.

An orientation of edges are also specified by normal vectors of edges, that are obtained from
the orientation of an edge by rotating it counterclockwise by $90$ degrees as in Figure~\ref{Y}. 

The {\it sign} $\epsilon (\vec{a}, v) $ of an oriented edge $\vec{a}$ incident to a vertex $v$  is defined,
following \cite{IIJO},  by
$$
\epsilon (\vec{a}, v) =\begin{cases}
1  & {\rm if \ the \ orientation \ of } \ \vec{a}\  {\rm  points \  into } \ v  , \\ 
-1 & {\rm otherwise}. 
\end{cases}
$$
In Figure~\ref{Y} (A), three edges are labeled by $a_i$, $i=1,2,3$. 
In (B), all-in orientations are specified to give oriented edges $\vec{a}_i$,
and in this case, we have $\epsilon (\vec{a}_i, v) =1$ for $i=1,2,3$. 
In (C), the orientations of the  top two arcs are reversed to give $\vec{a}_i \, '$ for $i=1,2$.
The labels $g,h,k$ are used below.

\begin{figure}[htb]
\begin{center}
\includegraphics[width=2.5in]{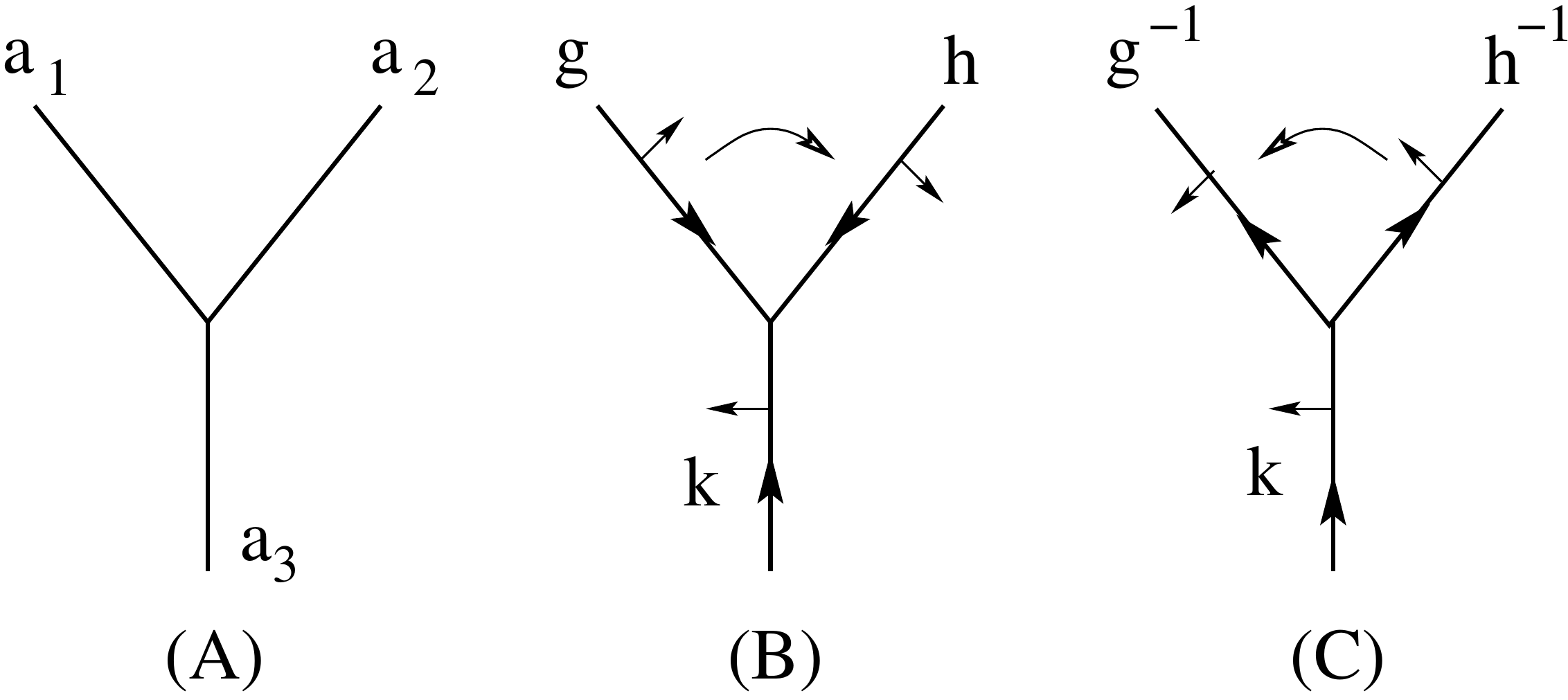}
\end{center}
\caption{Some orientation possibilities at a vertex}
\label{Y}
\end{figure}

A {\it coloring} of $\vec{D}$ by an augmented rack $(X,G)$ with a good involution $\rho$ is a 
${\mathcal C} : {\mathcal A}(\vec{D}) \rightarrow X$,
such that the coloring conditions 
described below are satisfied at every crossing and vertex.
In the figures, the letter assigned to each directed 
arc $\vec{a}$ represents an element of $X$ that is the image
${\mathcal C}(\vec{a}) \in X$ called a {\it color} of the directed arc $\vec{a}$. 

$(C1)$ First, we require that if a color of an oriented arc $\vec{a}$ is $x \in X$, then 
the color of the same edge with reversed orientation $\vec{a}\, '$ is $\rho(x)$.

$(C2)$ At each crossing, the coloring condition is as depicted in Figure~\ref{fig:positivenegative}.

$(C3)$ At each vertex $v$, if the colors on the incident edges $\vec{a}, \vec{b}, \vec{c}$ 
in clockwise are $x,y,z \in X$,
respectively, then it is required that 
$$ \nu(x)^{\epsilon (\vec{a} , v)}  \nu(y)^{\epsilon (\vec{b}, v)}  \nu(z)^{\epsilon (\vec{c}, v)} =e,$$
the identity element of $G$.

In Figure~\ref{Y} (B), the oriented edges are labeled by $\nu(x)=g, \nu(y)=h$ and $ \nu(z)=k$.
With these orientations in the figure, we have $ghk=e$ from (C3).
In (C), the orientations of the  top two arcs are reversed, so that the colors are changed 
for $\vec{a}_1\, '$ and  $\vec{a}_2\, '$ are changed to $\rho(x)$ and $\rho(y)$, respectively, 
and labeled by $\nu (\rho(x) )=g^{-1} $ and $\nu ( \rho(y)) = h^{-1}$.

The set of colorings of $\vec{D}$ by an augmented rack $(X, G) $ with a good involution $\rho$ is simply denoted by 
${\rm Col}_X (\vec{D})$. 

The following was mentioned in \cite{ClaSa}.

\begin{lemma}\label{lem:bij}
Let $(X, G)$ be an augmented rack.
If $X$ is finite and connected, then the fibers $\nu^{-1} (g)$ are in bijection over the image 
$g \in {\rm Im}(\nu)$ of $\nu$. 
\end{lemma}

\begin{proof}
Let $g, h \in {\rm Im}(\nu)$, and we establish a bijection between $\nu^{-1} (g)$ and $\nu^{-1} (h)$. 
Let $x \in \nu^{-1} (g)$ and $y \in \nu^{-1} (h)$.
Since $X$ is finite and connected, 
there exists  a sequence $x_0,  x_1, \ldots, x_n \in X$ such that 
$x=x_0$, $x_n=y$ and $x_{i+1}=x_{i-1} * x_i$ for all $i=1, \ldots, n-1$.
Set $k=\nu (x_1) \cdots \nu (x_{n-1})$, then $y=x \cdot k$, and 
$h= \nu(y)=\nu( x \cdot k ) = k^{-1} \nu(x) k = k^{-1} g k $. 
Since $ * x_i$ is a bijection on $X$ for each $i$, $\cdot k$ defines a bijection on $X$. 
For any $z \in \nu^{-1} (g)$, we have $\nu (z \cdot k )=k^{-1} \nu (z) k =k^{-1} g k = h$, 
hence the image of $\cdot k$ is in $\nu^{-1} (h)$.
Therefore $\cdot k$ defines a bijection $\nu^{-1} (g) \rightarrow \nu^{-1} (h)$.
\end{proof}

\begin{theorem}\label{thm:inv}
Let $(X,  G )$ be a finite  augmented connected rack with a good involution $\rho$.
Let $\vec{D}$ be an oriented  diagram of a trivalent spatial graph.
Then the set of colorings ${\rm Col}_X( \vec{D})$ of $\vec{D}$ by $X$ 
is in bijection  under  each of the moves in Figure~\ref{moves}, and therefore, 
its cardinality is an invariant of surface ribbons.
\end{theorem}

\begin{proof}
First we show that the set of colorings is in bijection under reversal of  the orientation 
of an edge $\vec{a}$. Let ${\cal C}$ be a given coloring of an orietned graph diagram $\vec{D}$. 
Let  $\vec{a}$ be an oriented edge, and $\vec{a}\, '$ the same edge with reversed orientation.
Then by the coloring condition $(C1)$, if ${\mathcal C} (\vec{a}) = x$, then 
${\mathcal C} ( \vec{a} \, ' ) $ is uniquely defined to be $ \rho (x)$.
For bijection, we show that this assignment, together with the colors already assigned by ${\mathcal C}$ on all the other edges satisfy the remaining conditions $(C2)$ and $(C3)$.
The condition $(C2)$ is shown in \cite{IIJO}. 
For the condition  $(C3)$, assume $x$ assigned to an arc whose orientation is reversed.
Then $x$ is changed to $\rho(x)$, 
and the sign is reversed to $\epsilon(\vec{a}\, ', v)=- \epsilon(\vec{a}, v)$, so that 
$\nu(x)^{\epsilon(\vec{a}, v)}$ is replaced by  
$$\nu(\rho(x) )^{\epsilon(\vec{a}\, ', v)}=[ \nu(x)^{-1} ]^{[ - \epsilon(\vec{a}, v)] } =\nu(x)^{\epsilon(\vec{a}, v)} ,$$
hence, in fact, it stays unchanged. The other cases are similar.

\begin{figure}[htb]
\begin{center}
\includegraphics[width=3in]{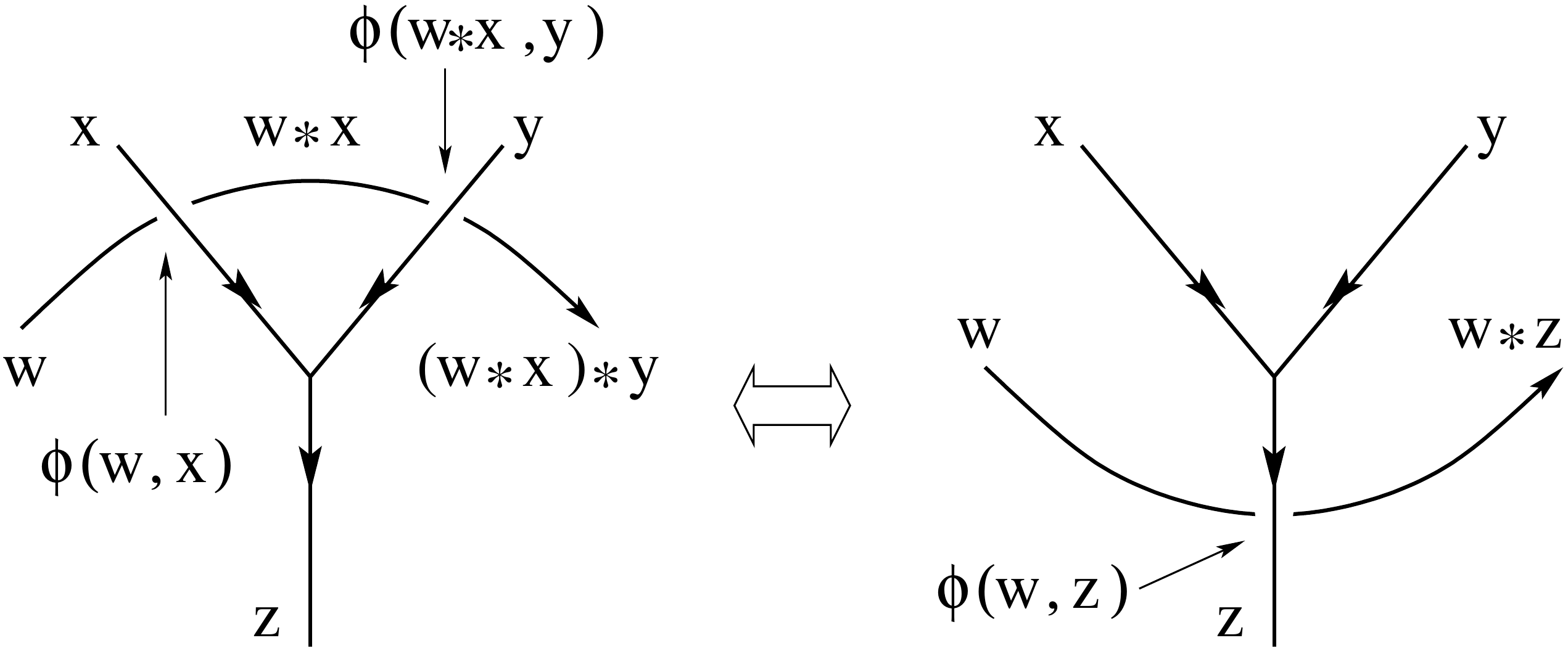}
\end{center}
\caption{Colored and weighted YI move}
\label{YIrack}
\end{figure}

The bijectivity of colorings under the moves that do not involve trivalent vertices are similar to
\cite{IIJO}. 
We check the remaining moves, i.e. we verify the YI, IY and IH moves of Figure~\ref{moves}. Move YI with the colorings is depicted in Figure~\ref{YIrack}. 
Let us denote the colors determined by $\mathcal C$ as in the LHS of Figure~\ref{YIrack}, where 
we do not consider the weights $\phi$ at this moment. 
From the definition of $*$ in an augmented rack, we have that $(w*x)*y = (w\nu(x))*y = w(\nu(x)\nu(y)) = w\nu(z) = w*z$, where in the penultimate equality we have used the assumption that $\nu(x)\nu(y) = \nu(z)$ at a trivalent vertex. It follows that there is a unique map $\mathcal C'$ that gives a compatible color on the arc labeled $w*z$ on the RHS of Figure~\ref{YIrack}. 

\begin{figure}[htb]
\begin{center}
\includegraphics[width=3in]{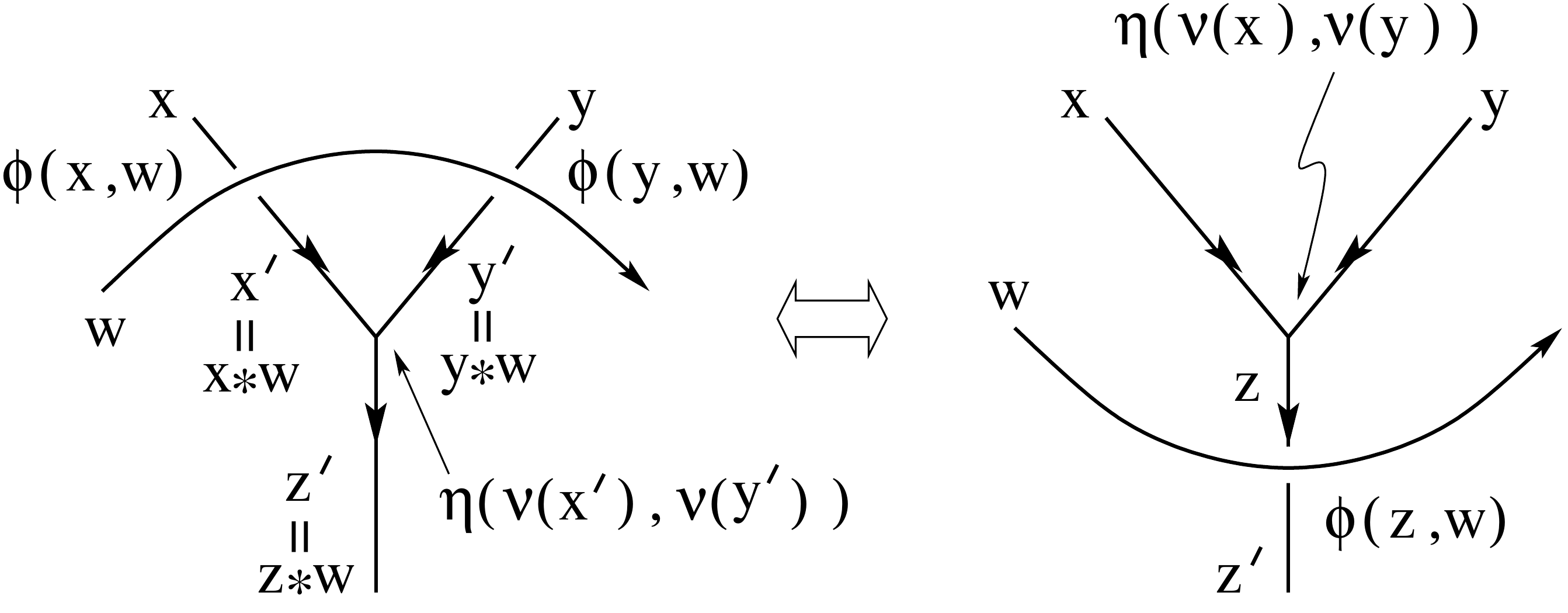}
\end{center}
\caption{Colored and weighted IY move}
\label{IYrack}
\end{figure}

Let us consider now the IY move of Figure~\ref{moves},
with a coloring $\mathcal C$ in LHS of 
 Figure~\ref{IYrack}. 
 We again do not consider the weights represented in the picture. 
 The color at the bottom arc specified by ${\mathcal C}$ is $z'$. 
 The two arcs colored by $x$ and $y$ underpass the arc $w$ and are  colored by labels $x*w$ and $y*w$. From the definition of $\mathcal C$ it follows that the color $z'$ of the remaining edge 
 satisfies 
 $\nu(z') = \nu(x*w)\nu(y*w) $.

 One computes 
 $\nu(z') =\nu(x*w)\nu(y*w) = \nu(x\cdot \nu(w))\nu(y\cdot \nu(w)) = \nu(w)^{-1}\nu(x)\nu(y)\nu(w) $.
 Set $z= z' \bar{*} w$, then $z'=z*w$. 
 One computes 
 $\nu(z')= \nu(z*w)=\nu(z\cdot \nu(w))= \nu(w)^{-1}\nu(z)\nu(w)$, hence we obtain 
 $\nu(x)\nu(y)=\nu(z)$. 
 On the RHS, given the color $z'$ at the bottom, the color $z$ in the figure is required to 
 satisfy $z'=z*w$. Hence this color $z$ is uniquely determined from $z'$ and $w$,
 and satisfies the coloring rule $(C3)$, yielding a unique coloring for the RHS.
 

\begin{figure}[htb]
\begin{center}
\includegraphics[width=2.5in]{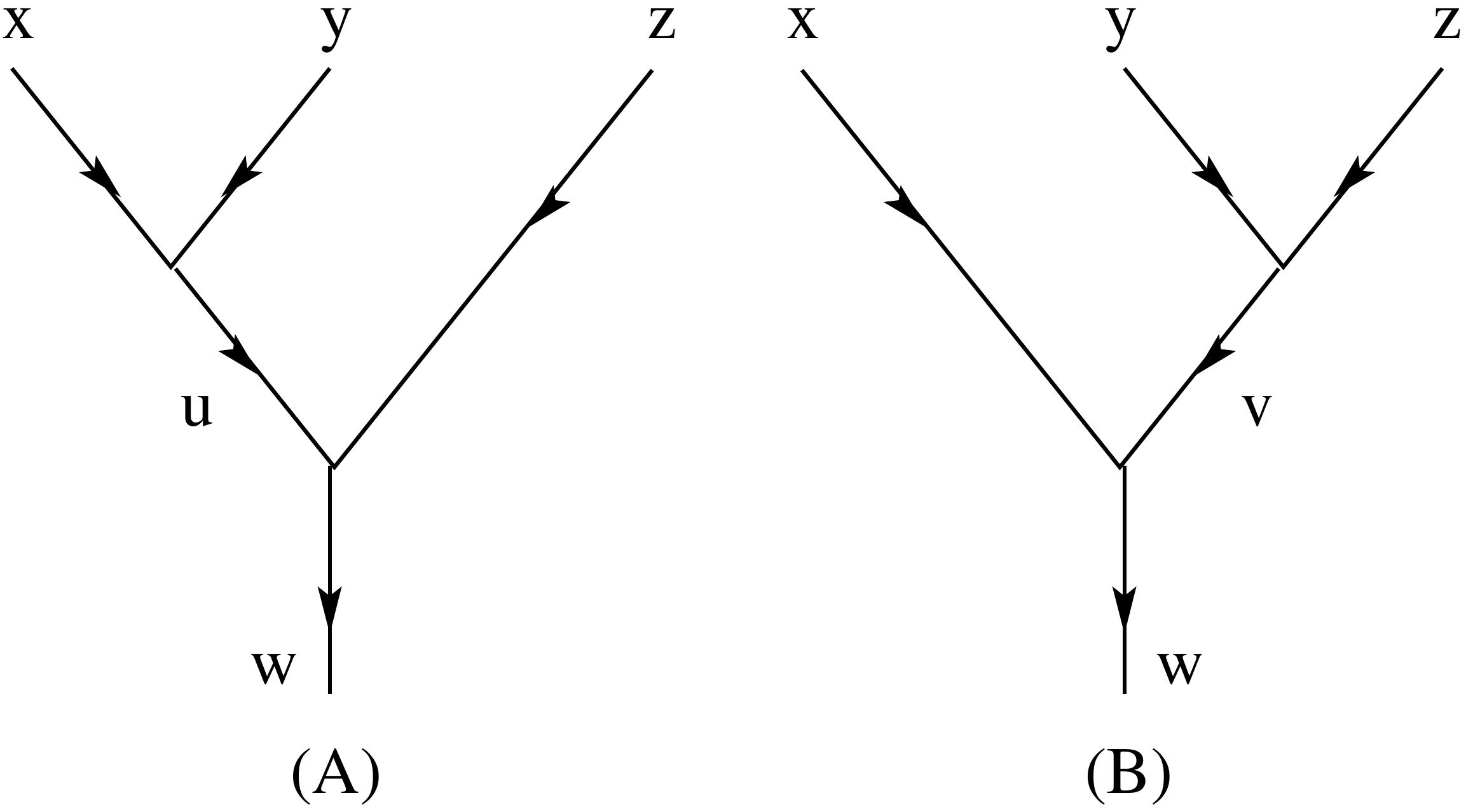}
\end{center}
\caption{Associativity and colors}
\label{IH1}
\end{figure}

 Let us now consider the IH move of Figure~\ref{moves}. The coloring condition corresponds to associativity in $G$, because on the LHS we have $\nu(u) = \nu(x)\nu(y)$ and $\nu(w) = \nu(u)\nu(z)$, while on the RHS we have $\nu(v) = \nu(y)\nu(z)$ and $\nu(w) = \nu(x)\nu(v)$. 
 This situation is depicted in Figure~\ref{IH1}, where $\nu$ is abbreviated.
 For fixed $x, y, z, w$, the possible choices of $u$ that determine the colorings  $\mathcal C$ are determined by all the $u\in X$ such that $u\in \nu^{-1}(\nu(x)\nu(y))$, while on the RHS the colorings $\mathcal C'$ are determined by the maps such that $v \in \nu^{-1}(\nu(y)\nu(z))$. 
 By Lemma~\ref{lem:bij}, these sets are in bijection.
\end{proof}



\subsection{Cocycle invariants from $G$-rack extensions}\label{subsec:Grackinv}

In this section we show that the cocycle invariant can be defined using 
rack cocycles corresponding to $G$-rack extensions defined in Section~\ref{subsec:Xext} when the cocycle satisfies  an additional requirement, which we now define.

\begin{definition}
{\rm
 Let $X$ denote an augmented rack with $\nu: X \rightarrow G$.
	  Let $\phi \in Z^2_{R}(X,A)$ be a rack $2$-cocycle of $X$ with coefficients in the abelian group $A$. 
Then $\phi$ is called {\it pre-additive} if for all $x,y,z \in X$ such that 
$\nu(x)\nu(y)=\nu(z)$,  it holds that $\phi(x,w)+\phi(y,w)=\phi(z,w)$.
}
\end{definition}

\begin{remark}
{\rm
Direct computations show that it is not always the case that a coboundary $\phi = \delta_R \xi$ 
satisfies the pre-additivity.
}
\end{remark}

 Let $X$ denote an augmented rack with $\nu: X \rightarrow G$
	with a good involution $\rho$.
	  Let $\phi$ be a rack $2$-cocycle of $X$ with coefficients in the abelian group $A$,
	  that is fibrant-additive $\phi \in Z^2_{RF+}(X,A)$ and pre-additive.

	  Let $S$ be a surface ribbon, and let us denote by $\vec D$ a diagram of $S$ with some 
	 orientation.  For each coloring  $\mathcal C$ 
	 of $\vec D$, we define the following Boltzmann weight. At each crossing $\tau$ as in Figure~\ref{fig:positivenegative}, we set $B_{\phi}(\mathcal C,\tau) = \phi(x,y)^{\sigma(\tau)}$, where $\sigma(\tau)$ denotes the sign of $\tau$ according to the convention estabilished in Figure~\ref{fig:positivenegative}. Each vertex $v$ of $\vec D$ does not receive a weight.

 \begin{definition}
	 	{\rm 
	 	Let $X$ be an augmented rack with $\nu: X \rightarrow G$ and with a good involution $\rho$, and let $\phi \in Z^2_{\rm RF+}(X,A)$ with coefficients in the abelian group $A$.
		Assume that $\phi$ is pre-additive and 
	 symmetric with respect to $\rho$.
		Let $S$ be a connected surface ribbon and let $\vec D$ denote a diagram of $S$. Then the cocycle invariant of $S$ is defined as
	 	$$
	 	\Psi_{\phi}(S) = 
		\sum_{\mathcal C} 
		\prod_{\tau} B_{\phi}(\mathcal C, \tau)
	 	$$
	 	where the sum runs over all the colorings, 
		and the products are taken over all the crossings of $\vec D$. If $S$ consists of multiple connected components, then one defines the same partition function for each connected component, and $\Psi_{\phi}(S)$ is understood to indicate the tuples containing all the partition functions corresponding to each connected component. 
 		}
	 \end{definition}
 
 \begin{proposition}\label{prop:Xextcocyinv}
 	The cocycle invariant above is well defined with respect to the choice of 
	oriented diagram $\vec D$. Therefore it is an invariant of the isotopy class of the surface ribbon $S$. 
 \end{proposition}
\begin{proof}
We check the invariance under each move as usual.
The invariance under the moves that do not involve trivalent vertices are proved in the same manner as the original quandle cocycle invariant \cite{CJKLS}.

Invariance under the YI move and IY move follow from the additivity of $\phi$ and the pre-additivity
condition as Figures~\ref{YIrack} and \ref{IYrack} indicates, respectively, and as the relation between the moves.

For the IH move, the bijection between the set of colorings involve the possible colors for $u$ in LHS and $v$ in RHS in Figure~\ref{IH1} as explained in Proof of Theorem~\ref{thm:inv}. 
Since all possible choices for $u$ are from the fiber 
$\nu^{-1} (\nu(x)\nu(y))$, the value of the cocycle evaluation involved remains unchanged, 
and so does for $v$ in RHS. 
Hence the state sum stays invariant.
\end{proof}

\subsection{Cocycle invariants from simultaneous augmented rack extensions}

In this section we define the cocycle invariant using 2-cocycles for the simultaneous extensions 
of augmented racks.
First we recall the following. 	
 
	 \begin{lemma}\cite{CJKLS} \label{lem:gpsymmetry}
		{\rm
			Let $G$ be a group, $A$ an abelian group, and $\eta \in Z^2_G(G, A)$ a 
			 normalized 
			$2$-cocycle.
			Then $\eta$  satisfies 
		\begin{itemize}
				\item[(1)]
			$\eta(g,h)=  - \eta(  gh, h^{-1}) = - \eta(g^{-1} , gh)$ for all $g, h \in G$,
			called the triangle symmetry.
				\item[(2)]
			$\eta(g,h) = - \eta(k,k^{-1}g) + \eta(k^{-1}g,h) + \eta(k,k^{-1}gh)$, for all $g,h,k\in G$. 
		\end{itemize}
				}
	\end{lemma}

\begin{figure}[htb]
\begin{center}
\includegraphics[width=2.5in]{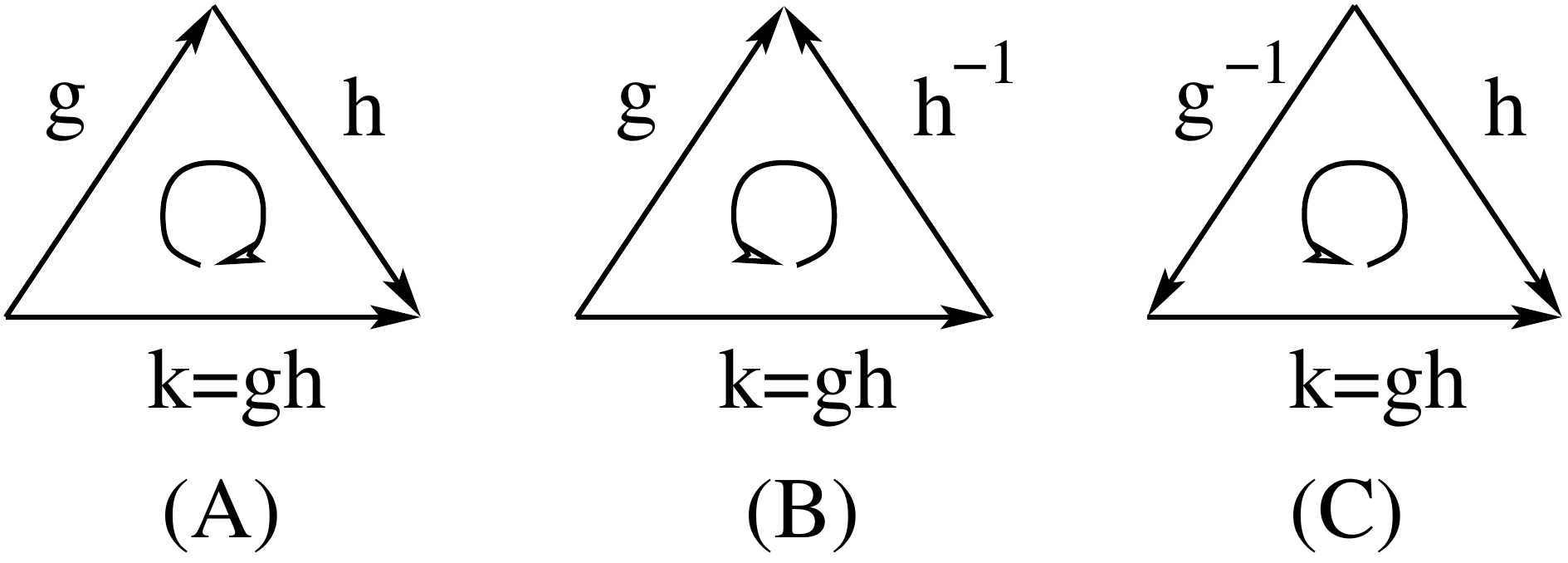}
\end{center}
\caption{Orientation changes in a triangle and triangle symmetry of group 2-cocycles}
\label{triangle}
\end{figure}

In Figure~\ref{triangle}, diagrammatic representations of the equalities (1) in Lemma~\ref{lem:gpsymmetry}  are depicted. A coherent (consecutive) orientations of two edges 
are evaluated for a 2-cocycle, and give an orientation of the triangle as depicted in (A). 
Orientation of one of these two edges are reversed in (B) and (C), which reverses the orientation
of the triangle, as well as group elements to be evaluated, as depicted.
The orientation reversal of the triangle corresponds to the negative sign of the evaluated cocycle.

Recall that at a vertex $v$,  the sign,  $\epsilon(\vec{a}, v )$, of an oriented  edge incident to $v$ 
was defined by $\epsilon(\vec{a}, v )=1$  if $\vec{a} $ points toward $v$, or equivalently, the orientation normal to the given direction is clockwise, 
and otherwise $\epsilon(\vec{a}, v )=-1$.

	 Let $X$ denote an augmented rack with $\nu: X \rightarrow G$
	with a good involution $\rho$.
	  Let $\phi$ be a rack $2$-cocycle of $X$ with coefficients in the abelian group $A$, and let $\eta$ be a group $2$-cocycles with coefficients in $A$. 
	  We further assume that  $\phi$ and $\eta$ are considered in multiplicative notation, and that
	  $\phi$ is $\eta$-derived as defined in Definition~\ref{def:Eq_Xext}.
	  Let $S$ be a surface ribbon, and let us denote by $\vec D$ a diagram of $S$ with some 
	 orientation.  For each coloring  $\mathcal C$ 
	 of $\vec D$, we define the following Boltzmann weights. At each crossing $\tau$ as in Figure~\ref{fig:positivenegative}, let $x,y\in X$ be colors assigned to oriented arcs as depicted.
	 We set $B_{\phi,\eta}(\mathcal C,\tau) := \phi(x,y)^{\sigma(\tau)}$, where $\sigma(\tau)$ denotes the sign of $\tau$ according to the convention established in Figure~\ref{fig:positivenegative}. 
	 
	 For each vertex $v$ of $\vec D$, 
	 let $\vec{a}_i$, $i=1,2,3$, be the incident edges at $v$, labeled in clockwise order
	 around $v$. Set 
	 $ \sigma(v) = 
	{\epsilon(\vec{a}_1, v)}+{\epsilon(\vec{a}_2, v)}+{\epsilon(\vec{a}_3, v)}$.
	 Let ${\mathcal C} (a_1)=x^{\epsilon(\vec{a}_1, v)}$,  ${\mathcal C} (a_2)=y^{\epsilon(\vec{a}_2, v)}$, and ${\mathcal C} (a_3)=z^{\epsilon(\vec{a}_3, v)}$.
	 Here we take a convention that $x^{-1}$ represents $\rho(x)$ for $x\in X$. 
	 The coloring condition implies that $x^{\epsilon(\vec{a}_1, v)}y^{\epsilon(\vec{a}_1, v)}y^{\epsilon(\vec{a}_1, v)}=1$.
	 Then we set the weight at $v$ for a coloring ${\mathcal C}$ to be $B_{\phi,\eta} (\mathcal C, v) := \eta(\nu(x),\nu(y))^{\sigma(v)}$.
	 When the edges are ordered counterclockwise, 
	 $\vec{a}_i\, '$, $i=1,2,3$, then the weight is defined to be 
	 $B_{\phi,\eta} (\mathcal C, v) := \eta(\nu(y),\nu(x))^{- \sigma(v)}=
	 \eta( \ \nu ( {\mathcal C} (\vec{a_2}\, ') ) ,   \nu ( {\mathcal C} (\vec{a_1}\, ') )\ )^{- \sigma(v)} $.
 Results similar to the following can be found in \cite{CIST,IIJO}.

\begin{lemma}
The weight  $B_{\phi,\eta} (\mathcal C, v) = \eta(\nu(x),\nu(y))^{\sigma(v)}$ does not depend on 
the choices of orientations of the three edges incident to $v$, nor the two entries $(\vec{a}_1, \vec{a}_2) $.
\end{lemma}

\begin{proof}
Set $\nu(x)=g$, $\nu(y)=h$ and $\nu(z)=k$ and first assume that all edges point into $v$,
so that all signs are positive.
Then the coloring condition is written as 
$g h  k =e$, hence 
we have that $B_{\phi,\eta} (\mathcal C, v) =\eta( {\mathcal C} (\vec{a}_1), {\mathcal C }( \vec{a}_2) )=\eta( \eta(g,h))^{\sigma(v)}
=\eta(g,h)$. By Lemma~\ref{lem:gpsymmetry}, we have $\eta(g,h)=\eta(gh, h^{-1})^{-1}=\eta(g^{-1} , gh)^{-1}$ in multiplicative notation. Setting $g'=gh$ and $h'=h^{-1}$ 
in $\eta(g',h')=\eta(g'^{-1} , g'h')^{-1}$, 
  we obtain 
$$\eta (g', h') = \eta (g'^{-1}, g'h')^{-1}=\eta( (gh)^{-1}, g)^{-1}=\eta( k, g)^{-1},$$
and from  $\eta(g,h)=\eta(gh, h^{-1})^{-1}=\eta (g', h')^{-1}$, 
we obtain  $\eta(g,h)=\eta(k,g)=\eta( \nu(z), \nu(x))=\eta( {\mathcal C} (\vec{a}_3), {\mathcal C }( \vec{a}_1) )$. 
 By cyclic symmetry we obtain that 
$B_{\phi,\eta} (\mathcal C, v) $ does not depend on the choice of $(\vec{a}_1, \vec{a}_2), (\vec{a}_2, \vec{a}_3), (\vec{a}_3, \vec{a}_1)$. 

We also have $\eta(g,h)=\eta(gh, h^{-1})^{-1}=\eta(k^{-1}, h^{-1})^{-1}$, and together 
with $\eta(g,h)=\eta(h,k)$, we obtain $\eta(k^{-1}, h^{-1})^{-1}=\eta(h,k)$ for all $h,k$.
Hence we have $\eta(g,h)=\eta(h^{-1}, g^{-1})^{-1}$.
If the orientations of $\vec{a}_1$ and $\vec{a}_2$ are reversed, then by reading inverse colors counterclockwise from 
$\vec{a_2}\, '$ to $\vec{a_1}\, '$  we 
obtain that 
$$B_{\phi,\eta} (\mathcal C, v) 
= \eta( h^{-1}, g^{-1})^{-1}=\eta(g,h). 
$$
Since cyclic permutations of choices or arcs and reversing orientations of two out of three arcs
generate all symmetries, we obtain the claim.
\end{proof}

One of the conditions posed to define the cocycle invariant is pre-additivity.
We show that it is satisfied by $\eta$-derived rack 2-cocycles.
For this purpose we identify the $\eta$-derivability to the construction given in \cite{CJKLS}.


Let $G$ be a group equipped with the conjugation rack structure, and let $A$ be an abelian group.
In \cite{CJKLS}, it was shown that, for a group 2-cocycle $\eta \in Z^2_G(G,A)$,
the function $\phi' (g,h):= \eta( g,h) - \eta(h, h^{-1}gh)$ is a rack 2-cocycle.
When $X=G$ is the conjugation quandle, it is regarded as an augmented quandle by 
$\nu={\rm id}_G$. Then  the $\eta$-derived 2-cocycle $\phi$ is regarded as an element of $Z^2_R(X,A)$. 

\begin{lemma}\label{lem:CJKLS}
Let $G$ be a group, also regarded as a conjugation rack, and $A$ be an abelian group as above.
If $\eta \in Z^2_G(G,A)$ is normalized, then the two rack 2-cocycle constructions coincide:
$$ \phi' (g,h):= \eta( g,h) - \eta(h, h^{-1}gh) = 
 \eta(h^{-1}, g) + \eta(h^{-1} g , h) =: \phi(x,y) .$$ 
\end{lemma}

\begin{proof}
By setting $z=y$ in Item (2) of Lemma~\ref{lem:gpsymmetry}, we have
$\eta( g,h) - \eta(h, h^{-1}gh)= \eta(h^{-1} x, y ) - \eta(h, h^{-1} g)$. 
By the first equality in Item (1), we have $ \eta(h, h^{-1} g) = - \eta(y^{-1} , x)$, as desired.
\end{proof}

\begin{lemma}\label{lem:additive}
Let $(X,G)$ be an augmented rack, and $A$ an abelian group.
If $\phi \in Z^2_{RF}(X,A)$ is $\eta$-derived for some $\eta\in Z^2_G(G,A)$
which is normalized,
then $\phi $ is additive.
\end{lemma}

\begin{proof}
Since $\phi$ is totally fibrant if 
$\phi$ is $\eta$-derived by Remark~\ref{rem:totallyfib}, 
 we use group elements 
for the variables for $\phi$, so that we denote $\phi(x,y)$, $x,y \in X$ by 
$\phi(g,h)$, where $g=\nu(x)$ and $h=\nu(y)$, $g,h \in G$. 
To show the additivity $\phi(k, g)+\phi(g^{-1}kg, h)=\phi(k, gh)$, 
we show 
$$\phi(k, g)+\phi(g^{-1}kg, h) + \eta(g,h) =\phi(k, gh) + \eta(g,h) .$$
By substituting the formula for $\phi'$ in Lemma~\ref{lem:CJKLS} instead of defining formula of $\eta$-derivability, we compute
\begin{eqnarray*}
\lefteqn{ \phi(k, g)+\phi(g^{-1}kg, h) + \eta(g,h) }\\
&=&
[ \eta(k,g)-\eta(g, g^{-1}kg) ] + [
\underline{  \eta( g^{-1}kg, h) - \eta( h, h^{-1} g^{-1} kgh)}
 ] +  \eta(g,h) \\
&=&
[ \eta(k,g) \underline{ -\eta(g, g^{-1}kg)} ] + [
\underline{  \eta( g^{-1}kg,  g^{-1}kg h ) } - \eta( g^{-1}kg h ,  h^{-1} g^{-1}kg h ) 
 ] + \underline{  \eta(g,h) } \\
 &=&
\underline{ \eta(k,g) + \eta(kg, g^{-1} k^{-1} gh)  }
- \eta( g^{-1} k^{-1} gh, h^{-1} g^{-1}kg h ) \\
&=&
\eta(  k, k^{-1} gh) +
\underline{ \eta (  g, g^{-1} k^{-1} g h ) - \eta( g^{-1} k^{-1} gh, h^{-1} g^{-1}kg h ) } \\
&=&
\underline{ \eta ( k, k^{-1} gh) + \eta(k^{-1} gh ,  h^{-1} g^{-1}kg h ) } + \eta(g,h) \\
&=&
\phi(k, gh) + \eta(g,h) 
\end{eqnarray*}
as desired. In Figure~\ref{squares}, the sequence of applications of 2-cocycle condition
is represented. Shaded regions correspond to the squares dual to crossings, 
and triangles correspond to $\eta$ evaluated by group elements labeling edges of triangles.
Diagrammatic representations of group 2-cocycle condition in Figure~\ref{assoc} and 
that of the triangle symmetry in Figure~\ref{triangle} are used in this computation.
\end{proof}

\begin{figure}[htb]
\begin{center}
\includegraphics[width=5in]{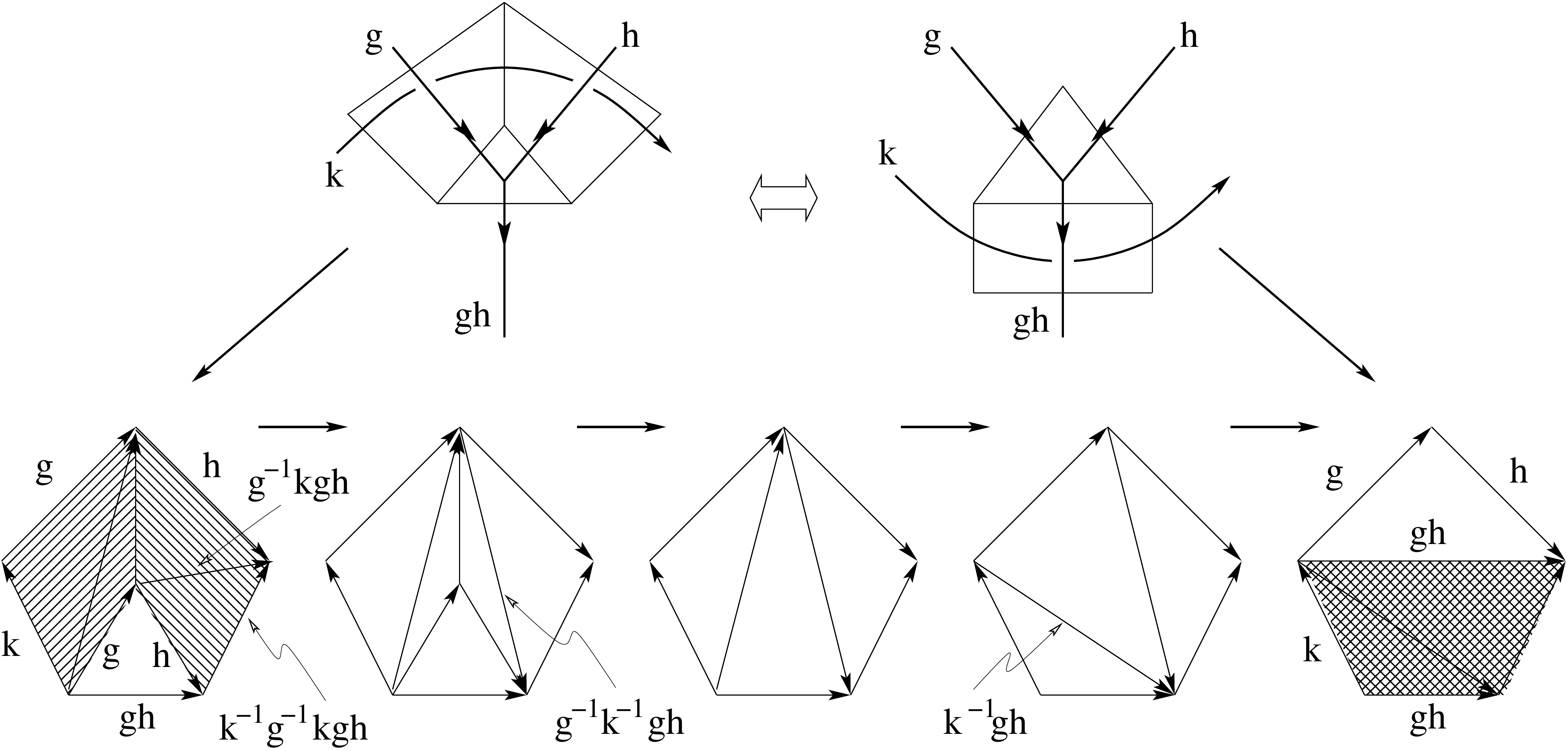}
\end{center}
\caption{Group 2-cocycles assigned to triangles and squares}
\label{squares}
\end{figure}

	 \begin{definition}
	 	{\rm 
	 	Let $X$ be an augmented rack with $\nu: X \rightarrow G$.
		Let $\phi \in Z^2_{\rm R}(X,A)$ be a rack 2-cocycle, and
		 let $\eta\in Z^2_{\rm G}(G,A)$ be a normalized group 2-cocycle,
		 such that such that $\phi$ is $\eta$-derived.  
		Assume further that $\phi$ is additive and symmetric with respect to $\rho$.
		Let $S$ be a connected surface ribbon and let $\vec D$ denote a diagram of $S$. Then the cocycle invariant of $S$ is defined as
	 	$$
	 	\Psi_{\phi,\eta}(S) = 
		\sum_{\mathcal C} 
		\prod_{\tau} B_{\phi,\eta}(\mathcal C, \tau)\prod_v B_{\phi, \eta} (\mathcal C, v),
	 	$$
	 	where the sums run over all the colorings, 
		and the products are taken over all the crossings and vertices of $\vec D$. If $S$ consists of multiple connected components, then one defines the same partition function for each connected component, and $\Psi_{\phi, \eta}(S)$ is understood to indicate the tuples containing all the partition functions corresponding to each connected component. 
 		}
	 \end{definition}
 
 \begin{proposition}\label{prop:Gextinv}
 	The cocycle invariant above is well defined with respect to the choice of 
	oriented diagram $\vec D$. Therefore it is an invariant of the isotopy class of the surface ribbon $S$. 
 \end{proposition}
\begin{proof}
The invariance under the moves that do not involve trivalent vertices are proved in the same manner as the original quandle cocycle invariant \cite{CJKLS}, and the same argument as the proof 
of Proposition~\ref{prop:Xextcocyinv}. 

Invariance under the YI move  follow from the additivity of $\phi$ and the $\eta$-derivability of $\phi$ 
as Figures~\ref{YIrack} and \ref{IYrack} indicates, respectively, and as the relation between the moves and these conditions were discussed in Section~\ref{sec:rack_group}.
 In particular, additivity is shown in Lemma~\ref{lem:CJKLS}.

For the IH move, the bijection between the set of colorings involve the possible colors for $u$ in LHS and $v$ in RHS in Figure~\ref{IH1} as explained in Proof of Theorem~\ref{thm:inv}. 
Since all possible choices for $u$ are from the fiber 
$\nu^{-1} (\nu(x)\nu(y))$, the value of the cocycle evaluation involved remains unchanged, 
and so does for $v$ in RHS. 
Hence the state sum stays invariant.
\end{proof}

\begin{example}
	{\rm 
		Let us consider the $2$-cocycles $\phi$ constructed in Example~\ref{ex:central_extension}.
Then by Lemma~\ref{lem:additive}, $\phi$ is additive.
Hence by Proposition~\ref{prop:Gextinv},
$\phi$ defines a cocycle invariant.
Thus the cocycle invariant for surface ribbons can be defined from any augmented racks through 
central group extensions.
	
}
\end{example}

Although it is desirable to investigate these  invariants about the relations to other invariants 
and novel applications, these are beyond the scope of this paper, and it is left to future studies.

   \bigskip        	

%

\end{document}